\documentclass[12pt]{article}
\usepackage{amsmath,amsthm,amssymb,amsfonts,graphicx,wrapfig}
\newtheorem{theorem}{Theorem}
\newtheorem{lemma}{Lemma}

\newcommand{\EHR}{{\mathrm{EHR}}}
\newcommand{\PNF}{{\rm{PNF}}}
\newcommand{\ch}{{\mathrm{ch}}}
\newcommand{\q}{{\mathrm{q}}}
\newcommand{\NE}{{\mathrm{NE}}}
\newcommand{\defeq}{\mathrel{\mathop:}=}
\newcommand{\cl}{\mathrm{cl}}
\newcommand{\x}{\mathbf{x}}
\newcommand{\y}{\mathbf{y}}
\newcommand{\z}{\mathbf{z}}
\begin{document}

\begin{center}{\Large First order sentences about random graphs: small number of alternations~\footnote{M.E. Zhukovskii is financially supported by the Ministry of Education and Science of the Russian Federation (the Agreement 02.A03.21.0008) and grants of the Russian Foundation for Basic Research 16-31-60052 and 15-01-03530.}}
\end{center}

%\vspace{0.1cm}

\begin{center}{\large A.D. Matushkin, M.E. Zhukovskii~\footnote{Moscow Institute of Physics and Technology, Laboratory of Advanced Combinatorics and Network Applications% (9 Institutskiy per., Dolgoprudny, Moscow Region, 141701, Russian Federation)
; RUDN University
%(6 Miklukho-Maklaya st, Moscow, 117198, Russian Federation)
}}
\end{center}

\vspace{0.5cm}

\begin{center}
{\bf Abstract}
\end{center}

Spectrum of a first order sentence is the set of all $\alpha$ such that $G(n,n^{-\alpha})$ does not obey zero-one law w.r.t. this sentence. We have proved that the minimal number of quantifier alternations of a first order sentence with infinite spectrum equals 3. We have also proved that the spectrum of a first-order sentence with a quantifier depth 4 has no limit points except possibly the points $1/2$ and $3/5$.

\section{Previous results on zero-one laws}
\label{laws}

In this paper, we consider first order sentences about graphs (a signature consists of two predicates $\sim$ (adjacency) and $=$ (equality) of arity $2$)~\cite{Strange,Survey}. Recall that a {\it quantifier depth} $\q(\phi)$ of a formula $\phi$ is the number of quantifiers in the longest past of nested quantifiers in this formula. Let $G(n,p)$ be a binomial random graph~\cite{Bol,Janson} with $n$ vertices and the probability $p$ of appearing of an edge. We say that $G(n,p)$ {\it obeys zero-one law w.r.t. a first order sentence $\phi$} , if either a.a.s. (asymptotically almost surely) $G(n,p)\models\phi$, or a.a.s. $G(n,p)\models\neg(\phi)$.

Let $S(\phi)$ be the set of all $\alpha>0$ such that $G(n,n^{-\alpha})$ does not obey zero-one law w.r.t. $\phi$. This set is called {\it a spectrum of $\phi$}. In 1988~\cite{Shelah}, S.~Shelah and J.~Spencer proved that there are only rational numbers in $S(\phi)$ for any first order sentence $\phi$. In 1990~\cite{Spencer_inf}, J.~Spencer proved that there exists first order sentence with an infinite spectrum and the quantifier depth $14$. In his paper~\cite{Spencer_Thresholds_Via}, he also proved that, for a first order sentence $\phi$ with a quantifier depth $k$, $S(\phi)\cap(0,1/(k-1))=\varnothing.$ This result was strengthened by M.~Zhukovskii in 2012~\cite{Zhuk0}: $S(\phi)\cap(0,1/(k-2))=\varnothing$. In particular, for any first order sentence $\phi$ with the quantifier depth $3$, $S(\phi)\cap(0,1)=\varnothing$, and, for any first order sentence $\phi$ with the quantifier depth $4$, $S(\phi)\cap(0,1/2)=\varnothing$. Later~\cite{ZhukOstr}, it was proved that, for any first order sentence $\phi$, the set $S(\phi)\cap(1,\infty)$ is finite. In~\cite{Zhuk_inf}, a first order sentence with the quantifier depth $5$ and an infinite spectrum was obtained. This formula is given in the statement below.
\begin{theorem}
Let $m\in\mathbb{N}$, $\alpha=\frac{1}{2}+\frac{1}{2(m+1)}$ and $p=n^{-\alpha}$. Then the random graph $G(n,p)$ does not obey zero-one law w.r.t. the sentence
$$
 \phi=\exists x_1\exists x_2\,
 \left[\left(\exists x_3\exists x_4\,\left(\bigwedge_{1\leq i<j\leq 4}(x_i\sim x_j)\right)\right)
 \wedge(\varphi(x_1,x_2))\right],
$$
where
$$
\varphi(x_1,x_2)=\forall y_1\, ([y_1\sim x_1]\vee[y_1\sim x_2]\vee[\forall y_2(\neg[(y_2\sim x_1)\wedge(y_2\sim y_1)])]\vee
$$
$$
[\exists z\,(z\sim x_1)\wedge(z\sim x_2)\wedge(\forall u\,[(\neg[(u\sim z)\wedge(u\sim y_1)])\vee(u\sim x_1)\vee(u\sim x_2)])]).
$$
\label{4_or_5}
\end{theorem}
So, a minimal quantifier depth of a first order sentence with an infinite spectrum equals either $4$, or $5$.

Note that the maximal number of quantifier alternations over all sequences of nested quantifiers in $\phi$ equals $3$ (we call this value the {\it number of quantifier alternations of $\phi$}). It is essential that all the negations are aplied to atomic formulas only. A prenex normal form of $\phi$ with the quantifier depth $8$ is given below
\begin{equation}
 \tilde\phi=\exists x_1\exists x_2\exists x_3\exists x_4\forall y_1\forall y_2\exists z\forall u \,
 \left[\left(\bigwedge_{1\leq i<j\leq 4}(x_i\sim x_j)\right)
 \wedge(\tilde\varphi(x_1,x_2,y_1,y_2,z,u))\right],
\label{PNF_example}
\end{equation}
where
$$
\tilde\varphi(x_1,x_2,y_1,y_2,z,u)=[y_1\sim x_1]\vee[y_1\sim x_2]\vee[\neg((y_2\sim x_1)\wedge(y_2\sim y_1))]\vee
$$
$$
[(z\sim x_1)\wedge(z\sim x_2)\wedge(\neg[(u\sim z)\wedge(u\sim y_1)])]\vee[u\sim x_1]\vee[u\sim x_2].
$$
This raises the following questions.
\begin{enumerate}
\item What is the minimal quantifier depth of a first order sentence with an infinite spectrum, 4 or 5?
\item What is the minimal number of quantifier alternations of a first order sentence with an infinite spectrum, 3 or less?
\item What is the minimal quantifier depth of a first order sentence in a prenex normal form with an infinite spectrum, 4, 5, 6, 7 or 8?
\end{enumerate}

We partially answer these questions in Sections~\ref{S_spectra},~\ref{depth}. In Section 3 we state and prove some results on first order formulas, that are used in our answers. Section 2 is devoted to the limit probabilities of properties related to the presence of small subgraphs and extensions in the random graph.

\section{Existence and extension statements}

Let $\phi$ be a first order sentence in a prenex normal form. We call $\phi$ {\it an existence sentence}, if all quantifiers of $\phi$ equal $\exists$. We call $\phi$ {\it an extension sentence}, if the sequence of all quantifiers of $\phi$ equals $\forall\ldots\forall\exists\ldots\exists$. We say that an existence sentence expresses an existence property, and an extension sentence expresses an extension property. An asymptotical behavior of probabilities of the random graph existence and extension properties was widely studied in~\cite{Erdos,Bol_small,Rucinski,Spencer_ext}. We summarize this study in the result given below.

Let $G,H$ be two graphs such that $H\subset G$, $V(H)=\{a_1,\ldots,a_s\}$, $V(G)\setminus V(H)=\{b_1,\ldots,b_m\}$, $s,m\geq 1$. Let $\rho(H)$ be a maximal fraction $e(Q)/v(Q)$ over all subgraphs $Q\subset H$ ($\rho(H)$ is called {\it the maximal density} of $H$). Here $e(Q),v(Q)$ denote the numbers of edges and vertices in $Q$ respectively. Let $\rho(G,H)$ be a maximal fraction $(e(Q)-e(H))/(v(Q)-v(H))$ over all $Q$ such that $H\subset Q\subset G$. We say that a graph has {\it the $(G,H)$-extension property}, if, for any its distinct vertices $y_1,\ldots,y_s$, there exist distinct vertices $x_1,\ldots,x_m$ such that, for all $i\in\{1,\ldots,s\}$, $j\in\{1,\ldots,m\}$, $y_i\neq x_j$ and the adjacency relation $a_i\sim b_j$ implies the adjacency relation $y_i\sim x_j$. %, $e(Q\setminus H)\neq e(Q)-e(H)$.
\begin{theorem}
Let $\rho(H)\neq 0$, $p=n^{-\alpha}$. If $\alpha<1/\rho(H)$, then a.a.s. in $G(n,p)$ there is an induced copy of $H$. If $\alpha>1/\rho(H)$, then a.a.s. in $G(n,p)$ there is no copy of $H$.

Let $\rho(G,H)\neq 0$, $p=n^{-\alpha}$. If $\alpha<1/\rho(G,H)$, then a.a.s. $G(n,p)$ has the $(G,H)$-extension property. If $\alpha>1/\rho(G,H)$, then a.a.s. $G(n,p)$ does not have the $(G,H)$-extension property.
\label{existence_and_extension}
\end{theorem}
It is not difficult to see that Theorem~\ref{existence_and_extension} implies finiteness of spectra of all existence and extension sentences (see Section~\ref{S_spectra}).

The next step is to consider sentences in prenex normal form that have 2 alternations. We call $\phi$ {\it a double-extension sentence}, if the sequence of all quantifiers of $\phi$ equals $\forall\ldots\forall\exists\ldots\exists\forall\ldots\forall$ (the respective properties are called {\it double-extension} as well). An asymptotical behavior of probabilities of the random graph double-extension properties was studied in~\cite{Alon,Zhuk_ext}.

Let $W,G,H$ be three graphs such that $H\subset G\subset W$, $V(H)=\{a_1,\ldots,a_s\}$, $V(G)\setminus V(H)=\{b_1,\ldots,b_m\}$, $V(W)\setminus V(G)=\{c_1,\ldots,c_r\}$, $s\geq 0$, $r,m\geq 1$. Assume that in $W$ there are edges between each connected component of $W|_{\{c_1,\ldots,c_r\}}$ and $W|_{\{b_1,\ldots,b_m\}}$. Let $\mathcal{W}$ be a finite set of graphs such that all $W\in\mathcal{W}$ satisfy the above conditions (but $r$ depends on $W$). We say that a graph has {\it the $(\mathcal{W},G,H)$-double-extension property}, if, for any its distinct vertices $y_1,\ldots,y_s$, there exist distinct vertices $x_1,\ldots,x_m$ such that, for all $W\in\mathcal{W}$ and all distinct vertices $z_1,\ldots,z_{r(W)}$,
\begin{itemize}
\item for all $i\in\{1,\ldots,s\}$, $j\in\{1,\ldots,m\}$, $y_i\neq x_j$ and the adjacency relation $a_i\sim b_j$ implies the adjacency relation $y_i\sim x_j$,
\item either there exists $h\in\{1,\ldots,r(W)\}$ and $i\in\{1,\ldots,s\}$ such that $[z_h=y_i]\vee[(z_h\nsim y_i)\wedge(c_h\sim a_i)]$,

or there exist $h\in\{1,\ldots,r(W)\}$ and $j\in\{1,\ldots,m\}$ such that $[z_h=x_j]\vee[(z_h\nsim x_j)\wedge(c_h\sim b_j)]$.
\end{itemize}
\begin{theorem}
\begin{sloppypar}
Let, for all $W\in\mathcal{W}$, $\rho(W,G)>\rho(G,H)>0$ and $p=n^{-\alpha}$, $\alpha\in (1/\rho(W,G),1/\rho(G,H))$. Then a.a.s. $G(n,p)$ has the $(\mathcal{W},G,H)$-double-extension property.
\end{sloppypar}
\label{double-extension}
\end{theorem}
We have proved that Theorems~\ref{existence_and_extension},~\ref{double-extension} imply the finiteness of spectra of double-extension sentences (see Section~\ref{S_spectra} as well).

So, we generalize the well-known results about existence, extension and double-extension properties and prove that spectra of all first order sentences with at most $2$ quantifier alternations are finite.

\section{Logical preliminaries}

\subsection{Some notations}

Recall that {\it a rooted tree} $T_R$ is a tree with one distinguished vertex $R$, which is called {\it the root}. If $R,\ldots,x,y$ is a simple path in $T_R$, then $x$ is called a parent of $y$ and $y$ is called a child of $x$. The relation of being {\it a descendant} is the transitive and reflexive closure of the relation of being a child. If $v\in V(T_R)$, then $T_R[v]$ denotes the subforest of $T_R$ spanned by the set of all descendants of $v$ (children of $v$ are its roots).\\

For two first order formulas $\phi_1(x_1,\ldots,x_s),\phi_2(y_1,\ldots,y_s)$ ($s\in\{0,1,2,\ldots\}$), we say that they are (asymptotically) {\it equivalent} (and write $\varphi_1\cong\varphi_2$), if there exists $n\in\mathbb{N}$ such that for any graph $G$ on at least $n$ vertices and any its vertices $v_1,\ldots,v_s$ either $G\models(\phi_1(v_1,\ldots,v_s))\wedge(\phi_2(v_1,\ldots,v_s))$, or $G\models(\neg(\phi_1(v_1,\ldots,v_s))\wedge(\neg(\phi_2(v_1,\ldots,v_s)))$. We say that a set of graphs $C$ is {\it a (asymptotical) first order property of a graph}, if there exists a first order sentence $\phi$ and $n\in\mathbb{N}$ such that, for any $G$ on at least $n$ vertices, $G\in C$ if and only if $G\models\phi$ (in this case, we say that $\phi$ {\it expresses} $C$). % Obviously, equivalent sentences express the same first order property.

\subsection{Language $\mathcal{F}$}

It is easy to see that any first order formula (not necessarily sentence) is equivalent to a formula constructed of the following symbols: variables, relational symbols $\sim,=,\nsim,\neq$, conjunctions $\wedge$, disjunctions $\vee$ and quantifiers $\forall,\exists$. We denote the set of formulas in this language by $\mathcal{F}$.

Let us state a simple observation of formulas in $\mathcal{F}$.

\begin{lemma}
Let $Z\in\{\wedge,\vee\}$, $z_1,z_2\in\{\forall,\exists\}$. Then, for any two formulas $\varphi_1(x_1,\ldots,x_s)$, $\varphi_2(x_1,\ldots,x_m)\in\mathcal{F}$ (not necessarily with $s$ and $m$ free variables respectively),
$$
 [z_1 x_1\ldots z_1 x_s \,(\varphi_1(x_1,\ldots,x_s))]Z [z_2 x_1\ldots z_2 x_m \,(\varphi_2(x_1,\ldots,x_m))]\cong
$$
$$
 z_1 x_1\ldots z_1 x_s z_2 x_{s+1}\ldots z_2 x_{s+m} \, ([\varphi_1(x_1,\ldots,x_s)]Z[\varphi_2(x_{s+1},\ldots,x_{s+m})]).
$$
\label{two_PNF}
\end{lemma}

For a formula $\phi\in\mathcal{F}$, define its {\it nesting forest $F(\phi)$} in the following way.

\begin{itemize}

\item If $\phi$ is an atomic formula, then its nesting forest is an empty graph.

\item Consider a formula $\varphi(x)$. If it has an empty nesting forest, then the nesting forest of the formula $\phi=\exists x\,(\varphi(x))$ (the formula $\phi=\forall x\,(\varphi(x))$) is an isolated vertex labeled by $\exists$ (by $\forall$). This vertex is a trivial tree rooted in its only vertex. Otherwise, let $F(\varphi(x))=T^1_{t_1}\sqcup\ldots\sqcup T^m_{t_m}$, where $T^1_{t_1},\ldots,T^m_{t_m}$ are trees rooted in $t_1,\ldots,t_m$ respectively. Then the nesting forest of the formula $\phi=\exists x\,(\varphi(x))$ (the formula $\phi=\forall x\,(\varphi(x))$) is a tree obtained by adding a vertex $t$ (which is the root of this three) labeled by $\exists$ (by $\forall$) to $F(\varphi(x))$ and edges from $t$ to each of $t_1,\ldots,t_m$.

\item If $\phi=(\varphi_1)\wedge(\varphi_2)$ (or $\phi=(\varphi_1)\vee(\varphi_2)$), then $F(\phi)$ is the disjoint union of $F(\varphi_1)$, $F(\varphi_2)$.

\end{itemize}

Consider a formula $\phi(x_1,\ldots,x_s)\in\mathcal{F}$ and its nesting forest $F(\phi)=T^1_{t_1}\sqcup\ldots\sqcup T^m_{t_m}$ consisting of trees $T^1_{t_1},\ldots,T^m_{t_m}$ rooted in $t_1,\ldots,t_m$ respectively. Let $v$ be a vertex of $T^i_{t_i}$ for some $i\in\{1,\ldots,m\}$. Consider the forest $T^i_{t_i}[v]$. Let $V$ be the set of all vertices of $T^i_{t_i}$ such that $v$ is a descendant for each of them, $[V]:=V\cup\{v\}$. Obviously, $T^i_{t_i}|_{[V]}$ is the path $t_i\ldots v$. Each of the vertices of this path corresponds to a bound variable of $\phi$. Let $y_1,\ldots,y_r$ be these variables ($y_{i+1},y_i$ corresponds to a child and a parent respectively). Then $T^i_{t_i}[v]$ is the nested forest of a subformula $\varphi(x_1,\ldots,x_s,y_1,\ldots,y_r)$ of $\phi$. The formula $\varphi(x_1,\ldots,x_s,y_1,\ldots,y_r)$ is called {\it a nested subformula of $\phi$}, the forest $F(\varphi(x_1,\ldots,x_s,y_1,\ldots,y_r))$ is called {\it a nested subforest of $F(\phi)$}.

Note that the quantifier depth of $\phi$ is the length of the longest path starting in a root (we denote it by $\q(\phi)$). For a path in $F(\phi)$ starting in a root consider the number of labels alternations (the number of (unordered) pairs of neighbors $\forall\exists$ and $\exists\forall$). For example, the number of labels alternations of the path $\exists\forall\forall\exists\exists\forall$ equals $3$. The maximal number of labels alternations over all paths starting in roots of $F(\phi)$ is called {\it the number of quantifier alternations of $\phi$} (we denote it by $\ch(\phi)$).

\subsection{Normal forms}

A formula $\phi\in\mathcal{F}$ is in {\it prenex normal form} (PNF) (we also say that $\phi$ is {\it a $\PNF$ formula} or {\it a $\PNF$ sentence}), if $F(\phi)$ is a path (all quantifiers are in the beginning of the formula). We say that $\hat\phi$ is {\it a $\PNF$ of} $\phi$, if $\hat\phi\in\mathcal{F}$, $\hat\phi$ is in PNF and $\hat\phi\cong\phi$. It is known~\cite{Logic2}, that for any first order formula (which is not necessarily in $\mathcal{F}$) there exists an equivalent first order formula in PNF. This immediately implies that $\phi$ has a PNF.

The formula $\phi$ is in {\it no-equivalence prenex normal form} (NEPNF) (we also say that $\phi$ is {\it $\NE\PNF$ formula} or {\it a $\NE\PNF$ sentence}), if $\phi$ is in PNF, and is constructed as follows. Consider an arbitrary sequence $z=(z_1,\ldots,z_m)$ of symbols from $\{\forall,\exists\}$. Let a formula $\phi_1(x_1,\ldots,x_m)\in\mathcal{F}$ has no quantifiers and no relations $=$ and $\neq$. For each $j\in\{1,\ldots,m-1\}$ a formula $\phi_{j+1}(x_1,\ldots,x_m)$ is obtained from $\phi_j(x_1,\ldots,x_m)$ in the following way:
$$
\phi_{j+1}(x_1,\ldots,x_m)=(x_{j+1}\neq x_j)\wedge\ldots\wedge(x_{j+1}\neq x_1)\wedge(\phi_j(x_1,\ldots,x_m)),\quad
\text{if }z_{j+1}=\exists,
$$
$$
\phi_{j+1}(x_1,\ldots,x_m)=(x_{j+1}=x_j)\vee\ldots\vee(x_{j+1}=x_1)\vee(\phi_j(x_1,\ldots,x_m)),\quad
\text{if }z_{j+1}=\forall.
$$
Finally, $\phi=z_1x_1\ldots z_mx_m\,(\phi_m(x_1,\ldots,x_m))$. We say that $(\phi_1(x_1,\ldots,x_m),z)$ is {\it $\NE$-basis} of $\phi$.

\begin{lemma}
For any \PNF sentence $\phi\in\mathcal{F}$ there exists an $\NE\PNF$ sentence $\hat\phi\in\mathcal{F}$ with the same sequence of quantifiers such that $\phi\cong\hat\phi$.
\label{NEPNF}
\end{lemma}

{\it Proof.} Let $\phi=z_1x_1\ldots z_mx_m\,(\varphi(x_1,\ldots,x_m))$, where $z_1,\ldots,z_m$ is a sequence of symbols from $\{\forall,\exists\}$. Set $\hat\phi_1(x_1,\ldots,x_m)=\varphi(x_1,\ldots,x_m)$. For each $j\in\{1,\ldots,m-1\}$ a formula $\hat\phi_{j+1}(x_1,\ldots,x_m)$ is obtained from $\hat\phi_j(x_1,\ldots,x_m)$ in the following way.

First, $\hat\phi_{j+1}^0(x_1,\ldots,x_m)$ is obtained from $\hat\phi_j(x_1,\ldots,x_m)$ by assuming that all
$$
x_1=x_{j+1},\ldots,x_j=x_{j+1}
$$
are false, and all
$$
x_1\neq x_{j+1},\ldots,x_j\neq x_{j+1}
$$
are true. For any $i\in\{1,\ldots,j+1\}$, $\hat\phi_j^i(x_1,\ldots,x_m))$ is obtained from $\hat\phi_j(x_1,\ldots,x_m)$ by assuming that all
$$
x_1=x_{j+1},\ldots,x_{i-1}=x_{j+1},x_i\neq x_{j+1},x_{i+1}=x_{j+1},\ldots,x_j=x_{j+1}
$$
are false, and all
$$
x_1\neq x_{j+1},\ldots,x_{i-1}\neq x_{j+1},x_i=x_{j+1},x_{i+1}\neq x_{j+1},\ldots,x_j\neq x_{j+1}
$$ are true.

Second, if $z_{j+1}=\exists$, then
$$
\hat\phi_{j+1}(x_1,\ldots,x_m)=\hat\phi_j^0(x_1,\ldots,x_m)\vee
\left[\bigvee_{i=1}^j(\hat\phi_j^i(x_1,\ldots,x_j,x_i,x_{j+2},\ldots,x_m))\right].
$$
Otherwise,
$$
\hat\phi_{j+1}(x_1,\ldots,x_m)=\hat\phi_j^0(x_1,\ldots,x_m)\wedge
\left[\bigwedge_{i=1}^j(\hat\phi_j^i(x_1,\ldots,x_j,x_i,x_{j+2},\ldots,x_m))\right].
$$

Let $\hat\phi$ be the NEPNF formula with the NE-basis $(\hat\phi_m(x_1,\ldots,x_m),(z_1,\ldots,z_m))$. It is easy to see that $\phi\cong\hat\phi$. Both formulas have the same sequence of quantifiers. $\Box$\\

We will frequently use the following corollary.

\begin{lemma}
Let $\phi=\exists x\,(\varphi(x))\in\mathcal{F}$. Then there exists an $\NE\PNF$ formula $\hat\phi=\exists x\,(\hat\varphi(x))\in\mathcal{F}$ such that $\phi\cong\hat\phi$ and $\ch(\phi)=\ch(\hat\phi)$.
\label{same_changes_PNF}
\end{lemma}

{\it Proof.} Let $F$ be a nesting forest of a formula with the quantifier depth $q$. Moreover, let $F$ be a rooted tree with a root $t(F)$. Denote by $t^r_1(F),\ldots,t^r_{a(r,F)}(F)$ all the vertices of $F$ which are at the distance $r-1$ from $t(F)$, where $r\in\{1,\ldots,q\}$, $a(r,F)\in\{1,2,\ldots\}$. Obviously, $a(1,F)=1,$ $a(r,F)\geq 1$ for all $r\in\{2,\ldots,q\}$. Let $r$ be the first positive integer such that $a(r,F)>1$ (if there is no such $r$, then set $r=q+1$). Let
$$
\mu[F]=q+1-r.
$$
Note that if $F$ is a simple path with an end-point $t(F)$, then $\mu[F]=0$.\\

Consider a sentence $\phi=\exists x\,(\varphi(x))\in\mathcal{F}$ such that $\ch(\phi)=k$. By Lemma~\ref{NEPNF}, it is sufficient to prove that there exists a PNF sentence $\hat\phi=\exists x\,(\hat\varphi(x))\in\mathcal{F}$ such that $\phi\cong\hat\phi$ and $\ch(\hat\phi)=k$. If $\mu[F(\phi)]=0$, then we are done ($\hat\phi=\phi$). Suppose that $\mu[F(\phi)]=m\in\mathbb{N}$, and that for any formula $\zeta$ (not necessarily closed and with an arbitrary first quantifier) with $\mu[F(\zeta)]<m$ the existence of an equivalent PNF sentence with the same number of quantifier alternations and the same first quantifier is already proven.

Let $\phi=z_1 x_1 \ldots z_s x_s\,(\varphi(x_1,\ldots,x_s))$, where $s=q-\mu[F(\phi)]$, $z_1=\exists$, $z_2,\ldots,z_s\in\{\forall,\exists\}$, and the first symbol of $\varphi(x_1,\ldots,x_s)$ is not a quantifier. The formula $\varphi(x_1,\ldots,x_s)$ is a logical combination $L$ (disjunctions and conjunctions) of formulas
$$
\exists x\,(\hat\varphi_i(x_1,\ldots,x_s,x)),\quad\forall x\,(\hat\varphi^j(x_1,\ldots,x_s,x)).
$$
Let $I=\{1,\ldots,|I|\}$ be the set of all such $i$s and $J=\{1,\ldots,|J|\}$ be the set of all such $j$s. So,
$$
\varphi(x_1,\ldots,x_s)=L(\exists x\,(\hat\varphi_i(x_1,\ldots,x_s,x)), i\in I;\,\, \forall x\,(\hat\varphi^j(x_1,\ldots,x_s,x)), j\in J).
$$
Obviously, for all $i\in I,$ $j\in J$,
$$
\mu[F(\hat\varphi_i(x_1,\ldots,x_s,x))]<m,\quad\mu[F(\hat\varphi^j(x_1,\ldots,x_s,x))]<m.
$$
By the induction hypothesis, for all $i\in I,$ $j\in J$ there exist PNF formulas
$$
\exists x\,(\tilde\varphi_i(x_1,\ldots,x_s,x))\cong\exists x\,(\hat\varphi_i(x_1,\ldots,x_s,x)),
$$
$$
\forall x\,(\tilde\varphi^j(x_1,\ldots,x_s,x))\cong\forall x\,(\hat\varphi^j(x_1,\ldots,x_s,x)),
$$
such that
$$
\ch(\exists x\,(\tilde\varphi_i(x_1,\ldots,x_s,x)))=\ch(\exists x\,(\hat\varphi_i(x_1,\ldots,x_s,x))),
$$
$$
\ch(\forall x\,(\tilde\varphi^j(x_1,\ldots,x_s,x)))=\ch(\forall x\,(\hat\varphi^j(x_1,\ldots,x_s,x))).
$$
Let
$$
 \tilde\psi(x_1,\ldots,x_s)=
 L(\exists x\,(\tilde\varphi_i(x_1,\ldots,x_s,x)), i\in I; \,\,
 \forall x\,(\tilde\varphi^j(x_1,\ldots,x_s,x)), j\in J).
$$
Then the formulas $\phi$ and
$$
\psi=z_1 x_1\ldots z_s x_s \,(\tilde\psi(x_1,\ldots,x_s))
$$
are equivalent and have the same numbers of quantifier alternations. Moreover, $F(\psi)$ is a rooted tree with exactly one vertex with a degree greater than $2$. The distance between this vertex and the root $t(F(\psi))$ is $s-1$. Let the distance between this vertex and a vertex with the biggest distance from the root equal $r$. Let us construct a formula $\psi^0\cong\psi$ such that $\ch(\psi^0)=\ch(\psi)$, $F(\psi^0)$ is a rooted tree with at most one vertex with a degree greater than $2$, and the distance between this vertex (if it exists) and a vertex with the biggest distance from the root is less than $r$. Obviously, we get the target formula $\hat\phi$ after applying such a construction at most $r$ times.

For all $i\in I,$ $j\in J$ let us find positive integers $d_i,d^j$ such that
$$
\exists x^1\,(\tilde\varphi_i(x_1,\ldots,x_s,x^1))=\exists x^1\ldots\exists x^{d_i}\,(\tilde\psi_i(x_1,\ldots,x_s,x^1,\ldots,x^{d_i})),
$$
$$
\forall x^1\,(\tilde\varphi^j(x_1,\ldots,x_s,x^1))=\forall x^1\ldots\forall x^{d^j}\,(\tilde\psi^j(x_1,\ldots,x_s,x^1,\ldots,x^{d^j})),
$$
where the formulas $\tilde\psi_i(x_1,\ldots,x_s,x^1,\ldots,x^{d_i})$, $\tilde\psi^j(x_1,\ldots,x_s,x^1,\ldots,x^{d^j})$ either have no quantifiers, or $\forall,\exists$ are the quantifier symbols they begin from respectively. Set $D_I=\sum_{i\in I}d_i$, $D_J=\sum_{j\in J}d^j$. Without loss of generality, assume $z_s=\exists$.

By Lemma~\ref{two_PNF}, there exists a formula (if $z_s=\forall$, then this formula starts with $\forall$)
$$
\tilde\psi^0(x_1,\ldots,x_s)=
$$
$$
\exists x_{s+1}\ldots \exists x_{s+D_I}\forall x_{s+D_I+1}\ldots\forall x_{s+D_I+D_J}\,(\hat\psi(x_1,\ldots,x_{s+D_I+D_J}))\cong
\tilde\psi(x_1,\ldots,x_s)
$$
such that
$$
\ch(z_1x_1\ldots z_sx_s\,(\tilde\psi^0(x_1,\ldots,x_s)))=\ch(z_1x_1\ldots z_sx_s\,(\tilde\psi(x_1,\ldots,x_s))).
$$
Moreover, $F(z_1x_1\ldots z_sx_s\,(\tilde\psi^0(x_1,\ldots,x_s)))$ is a tree with exactly one vertex with a degree greater than $2$, and the distance between this vertex and a vertex with the biggest distance from the root is less than $r$. Finally, set
$$
\psi^0=z_1x_1\ldots z_sx_s\,(\tilde\psi^0(x_1,\ldots,x_s)).\quad \Box
$$

\subsection{Ehrenfeucht games}

%Its proof is based on some (nearly) standart arguments (see Lemma~\ref{Ehren_PNF}) about Ehrenfeucht game with $q$ moves and $k$ changes. We give a description of the game below.

We consider three modification of Ehrenfeucht game.

\begin{enumerate}

\item The game $\EHR(G,H,q)$ is played on graphs $G$ and $H$. There are two players (Spoiler and Duplicator) and a fixed number of rounds $q$. At the $\nu\mbox{-}$th round ($1 \leq \nu \leq q$), Spoiler chooses either a vertex $x_{\nu}$ of $G$ or a vertex $y_{\nu}$ of $H$ (which does not coincide with any of chosen vertices). Duplicator chooses a vertex of the other graph (which does not coincide with any of chosen vertices as well). At the end of the game, the distinct vertices $x_{1},...,x_{q}$ of $G$, $y_{1},...,y_{q}$ of $H$ are chosen. Duplicator wins if and only if the map $f(x_i)=y_i$, $i\in\{1,\ldots,q\}$, is an isomorphism of $G|_{\{x_1,\ldots,x_q\}}$ and $H|_{\{y_1,\ldots,y_q\}}$.

\item In the game $\EHR(G,H,q,\leq k)$, there are $q$ rounds as well. The only difference with the game $\EHR(G,H,q)$ is that Spoiler can alternate at most $k$ times (if in the $i$-th round Spoiler chooses a vertex, say, in $G$, and in the $i+1$-th round --- in $H$ (or vice versa), then we say that he {\it alternates}).

\item The most strict rules (for Spoiler) are in the game $\EHR(G,H,q,k)$. The only difference with the game $\EHR(G,H, q,\leq k)$ is that Spoiler must alternates exactly $k$ times.

\end{enumerate}

Our results on first order properties of random graphs are based on the following typical arguments on the connection between an elementary equivalence and Ehrenfeucht game.

\begin{lemma}
The following two properties are equivalent:
\begin{itemize}
\item[1)] Spoiler has a winning strategy in $\EHR(G,H,q)$;
\item[2)] there is $\phi\in\mathcal{F}$ with $\q(\phi)=q$ such that $G\models\phi$, $H\models\neg(\phi)$.
\end{itemize}
\label{Ehren_1}
\end{lemma}

This statement is a particular case of Ehrenfeucht theorem~\cite{Ehren}.%, its proof can be found, e.g., in~\cite{}.

The next two lemmas have typical proofs. We give it here for the sake of convenience.

\begin{lemma}
The following two properties are equivalent:
\begin{itemize}
\item[1)] Spoiler has a winning strategy in $\EHR(G,H,q,\leq k)$;
\item[2)] there is $\phi\in\mathcal{F}$ with $\q(\phi)=q$ and $\ch(\phi)\leq k$ such that $G\models\phi$, $H\models\neg(\phi)$.
\end{itemize}
\label{Ehren_2}
\end{lemma}

{\it Proof.} First, let us prove that 2) implies 1). Let $\phi\in\mathcal{F}$ be a sentence such that $\ch(\phi)\leq k$, $\q(\phi)=q$, $G\models\phi$ and $H\models\neg(\phi)$. We will describe a winning strategy of Spoiler by an induction on the number of played rounds. The sentence $\phi$ is a logical combination (disjunctions and conjunctions) of sentences $\varphi_i=\exists x\,(\hat\varphi_i(x))$ and $\varphi^j=\forall x\,(\hat\varphi^j(x))$. Obviously, one of these sentences is true for $G$ and not true for $H$. Let $\beta_1$ be the root of the nesting forest (tree) of this sentence. If, say,
$$
G\models\exists x\,(\hat\varphi_1(x)),\quad H\models\neg(\exists x\,(\hat\varphi_1(x))),
$$
then set $\varphi_1(x):=\hat\varphi_1(x)$. Spoiler in the {\bf first round} chooses a vertex $v_1$ such that $G\models\varphi_1(v_1)$. Duplicator chooses a vertex $u_1$. Obviously, $H\models\neg(\varphi_1(u_1))$. Denote the root of $F(\phi)$ by $\beta_1$. If, say,
$$
G\models\forall x\,( \hat\varphi^1(x)),\quad H\models\neg(\forall x\,(\hat\varphi^1(x))),
$$
then set $\varphi_1(x):=\hat\varphi^1(x)$. Spoiler in the first round chooses a vertex $u_1$ such that $H\models\neg(\varphi_1(u_1))$. Duplicator chooses a vertex $v_1$. Obviously, $G\models\varphi_1(v_1)$.

Fix $m\in\{2,\ldots,k\},\ell=\ell(m-1)\in\{1,\ldots,m-1\}$ and vertices $v_1,\ldots,v_{m-1}$, $u_1,\ldots,u_{m-1}$ (not necessarily distinct) in the graphs $G,H$ respectively. Suppose that $v_{i_1},\ldots,v_{i_{\ell}}$, $u_{i_1},\ldots,u_{i_{\ell}}$ are all distinct vertices of $v_1,\ldots,v_{m-1}$, $u_1,\ldots,u_{m-1}$ respectively. Moreover, $v_j=v_{i_r}$ if and only if $u_j=u_{i_r}$. Suppose that {\bf $\ell$ rounds are played}, and the vertices $v_{i_1},\ldots,v_{i_{\ell}}$, $u_{i_1},\ldots,u_{i_{\ell}}$ are chosen in the graphs $G,H$ respectively. Moreover, suppose that in $\phi$ there exists a nested subformula $\varphi_{m-1}(x_1,\ldots,x_{m-1})$ such that $\q(\varphi_{m-1}(x_1,\ldots,x_{m-1}))=q-m+1$,
$$
 G\models\varphi_{m-1}(v_1,\ldots,v_{m-1}),\quad H\models\neg(\varphi_{m-1}(u_1,\ldots,u_{m-1})).
$$
The formula $\varphi_{m-1}(x_1,\ldots,x_{m-1})$ is a logical combination (disjunctions and conjunctions) of formulas
$$
\exists x_m\,(\hat\varphi_i(x_1,\ldots,x_{m})),\quad \forall x_m\,(\hat\varphi^j(x_1,\ldots,x_{m})).
$$
Obviously, (at least) one of these formulas is true for $G$ on $v_1,\ldots,v_{m-1}$ and not true for $H$ on $u_1,\ldots,u_{m-1}$. Let $\beta_m$ be the root of the nesting forest of such a formula. If, say,
$$
G\models\exists x_m\,(\hat\varphi_1(v_1,\ldots,v_{m-1},x_m)),\quad H\models\neg(\exists x_m\, (\hat\varphi_1(v_1,\ldots,v_{m-1},x_m))),
$$
then we find a vertex $v_m$ such that $G\models\hat\varphi_1(v_1,\ldots,v_{m-1},v_m)$ and set $\varphi_m(x_1,\ldots,x_m)=\hat\varphi_1(x_1,\ldots,x_m)$. If $v_m\in\{v_1,\ldots,v_{m-1}\}$, then Spoiler ``skips'' this round, and we set $u_m=u_j$, where $j\in\{1,\ldots,m-1\}$ is a number such that $v_j=v_m$. Otherwise, Spoiler chooses a vertex $v_{i_{\ell+1}}=:v_m$ and Duplicator chooses a vertex $u_{i_{\ell+1}}=:u_m$. Obviously, in both cases, $H\models\neg(\varphi_m(u_1,\ldots,u_m))$. If, say,
$$
G\models\forall x_m\,(\hat\varphi^1(v_1,\ldots,v_{m-1},x_m)),\quad H\models\neg(\forall x_m\, (\hat\varphi^1(v_1,\ldots,v_{m-1},x_m))),
$$
then fix a vertex $u_m$ such that $H\models\neg(\hat\varphi^1(u_1,\ldots,u_{m-1},u_m))$ and set $\varphi_m(x_1,\ldots,x_m)=\hat\varphi^1(x_1,\ldots,x_m)$. If $u_m\in\{u_1,\ldots,u_{m-1}\}$, then Spoiler ``skips'' this round, and we set $v_m=v_j$, where $j\in\{1,\ldots,m-1\}$ is a number such that $u_j=u_m$. Otherwise, Spoiler chooses a vertex $u_{i_{\ell+1}}=:u_m$ and Duplicator chooses a vertex $v_{i_{\ell+1}}=:v_m$. Obviously, in both cases, $G\models\varphi_m(v_1,\ldots,v_m)$.

This strategy is winning for Spoiler in $\EHR(G,H,q)$. Moreover, it is easy to see that Spoiler alternates $\tilde k$ times, where $\tilde k\leq k$ is the number of labels alternations in the path $\beta_1\beta_{\ell(2)}\ldots\beta_{\ell(q)}$.\\

It remains to prove that 1) implies 2). Let Spoiler have a winning strategy in the game $\EHR(G,H,k,q)$ with a first move in $G$. Let us construct a sentence $\phi\in\mathcal{F}$ such that $\ch(\phi)=k$, $\q(\phi)=q$, $G\models\phi$ and $H\models\neg(\phi)$.

Let, after $q$ rounds, distinct vertices $v_1,\ldots,v_q$ in $G$ and $u_1,\ldots,u_q$ in $H$ be chosen. As Spoiler wins in $q$ rounds, there is a formula $\varphi_q(x_1,\ldots,x_q)\in\mathcal{F}$ such that $\q(\varphi_q(x_1,\ldots,x_q))=0$ and $G\models\varphi_q(v_1,\ldots,v_q)$, $H\models\neg(\varphi_q(u_1,\ldots,u_q))$.

Fix $m\in\{0,\ldots,q-1\}$. Let after $m$ rounds, distinct vertices $v_1,\ldots,v_m$ in $G$ and $u_1,\ldots,u_m$ in $H$ be chosen. In the $m+1$-th round, Spoiler chooses, say, a vertex $v_{m+1}\in V(G)$ (according to his winning strategy). Suppose that, for any choice of Duplicator (denote it by $u_{m+1}$), there is a formula $\varphi_{m+1}^{u_{m+1}}(x_1,\ldots,x_{m+1})\in\mathcal{F}$ such that $\q(\varphi_{m+1}^{u_{m+1}}(x_1,\ldots,x_{m+1}))=q-m-1$ and
$$
G\models\varphi_{m+1}^{u_{m+1}}(v_1,\ldots,v_{m+1}),\quad H\models\neg(\varphi_{m+1}^{u_{m+1}}(u_1,\ldots,u_{m+1})).
$$
Note that, for a fixed number of free variables, there is only a finite number of representatives of $\cong$-equivalence classes of formulas in $\mathcal{F}$ with a fixed quantifier depth (see, e.g.,~\cite{Logic2}). Therefore, there are a positive constant $C$ (which does not depend on $|V(G)|$, $|V(H)|$) and a set $\mathcal{U}\subset V(H)$ with $|\mathcal{U}|\leq C$ such that the following property holds. For any $u_{m+1}\in V(H)$, there exists $u\in\mathcal{U}$ such that $\varphi_{m+1}^u(x_1,\ldots,x_{m+1})\cong\varphi^{u_{m+1}}_{m+1}(x_1,\ldots,x_{m+1})$. Set
$$
\varphi_m(x_1,\ldots,x_m)=\exists x_{m+1}\,\left(\bigwedge_{u\in\mathcal{U}}(\varphi^u_{m+1}(x_1,\ldots,x_{m+1}))\right).
$$
%Let $\mathcal{U}_{\exists}$ be the set of all $u\in\mathcal{U}$ such that $\hat\varphi^u_{m+1}(x_1,\ldots,x_{m+1})$ starts from $\exists$. If $\mathcal{U}_{\exists}=\mathcal{U}$, then, by Lemma~\ref{two_PNF}, there exists a PNF formula
%\begin{equation}
%\hat\varphi_m(x_1,\ldots,x_m)\cong\exists x_{m+1}\,\left(\bigwedge_{u\in\mathcal{U}}(\hat\varphi^u_{m+1}(x_1,\ldots,x_{m+1}))\right)
%\label{glue}
%\end{equation}
%such that
%$$
%\ch(\hat\varphi_m(x_1,\ldots,x_m))=\max_{u\in\mathcal{U}}\hat\varphi^u_{m+1}(x_1,\ldots,x_{m+1}).
%$$
%If $\mathcal{U}_{\exists}=\varnothing$, then, by Lemma~\ref{two_PNF}, there exists a PNF formula~(\ref{glue}) such that
%$$
%\ch(\hat\varphi_m(x_1,\ldots,x_m))=1+\max_{u\in\mathcal{U}}\hat\varphi^u_{m+1}(x_1,\ldots,x_{m+1}).
%$$
%Finally, if $\mathcal{U}_{\exists}\notin\{\varnothing,\mathcal{U}\}$, then, for each $u\in\mathcal{U}\setminus\mathcal{U}_{\exists}$, we add a fictitious $\exists x$ in the beginning of $\hat\varphi^u_{m+1}(x_1,\ldots,x_{m+1})$ and apply Lemma~\ref{two_PNF}. So, there exists a PNF formula~(\ref{glue}) such that
%$$
%\ch(\hat\varphi_m(x_1,\ldots,x_m))=\max\left\{\max_{u\in\mathcal{U}_{\exists}}\hat\varphi^u_{m+1}(x_1,\ldots,x_{m+1}),1+\max_{u\in\mathcal{U}\setminus\mathcal{U}_{\exists}} \hat\varphi^u_{m+1}(x_1,\ldots,x_{m+1})\right\}.
%$$
Obviously, $G\models\varphi_m(v_1,\ldots,v_m)$ and $H\models\neg(\varphi_m(u_1,\ldots,u_m))$.

Finally, let Spoiler choose a vertex $u_{m+1}\in V(H)$ and, for any choice of Duplicator $v_{m+1}\in V(G)$,  there exists a formula $\varphi_{m+1}^{v_{m+1}}(x_1,\ldots,x_{m+1})\in\mathcal{F}$ such that $\q(\varphi_{m+1}^{v_{m+1}}(x_1,\ldots,x_{m+1}))=q-m-1$ and
$$
G\models\varphi_{m+1}^{v_{m+1}}(v_1,\ldots,v_{m+1}),\quad H\models\neg(\varphi_{m+1}^{v_{m+1}}(u_1,\ldots,u_{m+1})).
$$
As in the previous case, there are a positive constant $C$ (which does not depend on $|V(G)|$, $|V(H)|$) and a set $\mathcal{V}\subset V(H)$ with $|\mathcal{V}|\leq C$ such that the following property holds. For any $v_{m+1}\in V(G)$, there exists $v\in\mathcal{V}$ such that $\varphi_{m+1}^v(x_1,\ldots,x_{m+1})\cong\varphi^{v_{m+1}}_{m+1}(x_1,\ldots,x_{m+1})$. Set
$$
\varphi_m(x_1,\ldots,x_m)=\forall x_{m+1}\,\left(\bigvee_{v\in\mathcal{V}}(\varphi^v_{m+1}(x_1,\ldots,x_{m+1}))\right).
$$
By the induction, we get that $\phi=\phi_0$ is the required sentence which is true for $G$ and false for $H$. Obviously, $\ch(\phi)\leq k$. $\Box$.

%$\phi_{m+1}^1[v_{m+1}](x_1,\ldots,x_{m+1}),\phi_{m+1}^2[v_{m+1}](x_1,\ldots,x_{m+1})\in\mathcal{F}$ such that $\q(\phi_{m+1}^1[v_{m+1}](x_1,\ldots,x_{m+1}))=q-m-1$, $\phi_{m+1}^2[v_{m+1}](x_1,\ldots,x_{m+1})$ is in PNF and
%$$
%H\models(\phi_{m+1}^1[u_{m+1}](u_1,\ldots,u_{m+1}))\wedge(\phi_{m+1}^2[u_{m+1}](u_1,\ldots,u_{m+1})),
%$$
%$$ H\models(\neg(\phi_{m+1}^1[u_{m+1}](u_1,\ldots,u_{m+1})))\wedge(\neg(\phi_{m+1}^2[u_{m+1}](u_1,\ldots,u_{m+1}))).
%$$
%For $j\in\{1,2\}$, denote by $\Phi_{m+1}^j$ a finite set of formulae from $\mathcal{F}$ with $m+1$ variables and the quantifier depth $q-m-1$ such that I and II hold. Let $\phi_m^1,\phi_m^2\in\mathcal{F}$ be formulae such that $\q(\phi_m^1)=q-m$, $\phi_m^2$ is in PNF, and, for $j\in\{1,2\}$,
%$$
%\phi_m^j(x_1,\ldots,x_m)\cong\forall x_{m+1}\,\left(\bigvee_{\varphi\in\Phi_{m+1}^j}(\neg(\varphi(x_1,\ldots,x_{m+1})))\right),
%$$
%Obviously, for $j\in\{1,2\}$, $G\models\phi_m^j(v_1,\ldots,v_m)$, $H\models\neg(\phi_m^j(u_1,\ldots,u_m))$.

%By the induction, we get that $\phi^1=\phi_0^1$, $\phi^2=\phi_0^2$ are the required sentences which are true for $G$ and false for $H$. Obviously, $\ch(\phi^1_0)=k$. $\Box$.\\

\begin{lemma}
The following two properties are equivalent:
\begin{itemize}
\item[1)] Spoiler has a winning strategy in $\EHR(G,H,q,k)$;
\item[2)] there is $\phi\in\mathcal{F}$ with $\q(\phi)=q$ such that a number of labels alternations in any path of $F(\phi)$ on $q$ vertices starting in a root equals $k$, and $G\models\phi$, $H\models\neg(\phi)$.
\end{itemize}
\label{Ehren_3}
\end{lemma}

{\it Proof}. First, let us prove that 2) implies 1). The winning strategy of Spoiler is absolutely the same as in the proof of Lemma~\ref{Ehren_2}. The only thing we should prove is that Spoiler alternates exactly $k$ times. If $\ell(q)<q$, then consider a path $\beta_1\ldots\beta_{\ell(q)}\beta_{\ell(q)+1}\ldots\beta_q$ in $F(\phi)$. The number of labels alternations in this path equals $k$. Therefore, $k-\tilde k\leq q-\ell(q)$. So, Spoiler can choose graphs (and an arbitrary vertex) in each of the remaining rounds in a way such that he will alternate $k$ times overall. If $\ell(q)=q$, then, obviously, $\tilde k=k$.

It remains to prove that 1) implies 2). The formula $\phi$ is constructed in the same way as in the proof of Lemma~\ref{Ehren_2}. We only need to prove that $\ch(\phi)=k$. Consider an arbitrary path $\beta_1\ldots\beta_q$ in $F(\phi)$ starting in a root. Note that $\beta_i$ is labeled by $\exists$ if and only if there exists a Duplicator's strategy such that in the $i$-th round Spoiler chooses $G$. Therefore, the number of labels alternations in this path equals $k$. $\Box$\\

\section{Spectra of formulas with small numbers of alternations}
\label{S_spectra}

Let us start this section with the following simple observation.

\begin{lemma}
If $\phi\in\mathcal{F}$ and $\alpha\in S(\phi)$, then there exists an $\NE\PNF$ sentence $\hat\phi$ such that $\ch(\phi)=\ch(\hat\phi)$ and $\alpha\in S(\hat\phi)$ as well.
\label{reduce_to_one_root}
\end{lemma}

{\it Proof.} By Lemma~\ref{same_changes_PNF}, it is enough to prove that there exists a sentence $\hat\phi=\exists x\,(\varphi(x))\in\mathcal{F}$ such that $\alpha$ belongs to its spectrum.

As $\alpha\in S(\phi)$, there exist $\varepsilon>0$ and sequences $n_i,m_i$ such that, for any $i\in\mathbb{N}$,
$$
\min\left\{{\sf P}\left(G(n_i,n_i^{-\alpha})\models\phi\right),
{\sf P}\left(G(m_i,m_i^{-\alpha})\models\neg(\phi)\right)\right\}>\varepsilon.
$$
Fix $i\in\mathbb{N}$. Let $G,H$ be graphs on $n_i,m_i$ vertices respectively such that $G\models\phi$, $H\models\neg(\phi)$. The formula $\phi$ is a logical combination (disjunctions and conjunctions) of formulas
$$
\exists x\,(\varphi_j(x)),\quad\forall x\,(\varphi^j(x)).
$$
Let $N$ be the number of all formulas in this combination. Obviously, there exists either $j$ such that $G\models\exists x\,(\varphi_j(x))$, $H\models\neg(\exists x\,(\varphi_j(x)))$ or $j$ such that $G\models\forall x\,(\varphi^j(x))$, $H\models\neg(\forall x\,(\varphi^j(x)))$. Therefore, there exists $\varphi(x)=\varphi(x,i)\in\mathcal{F}$ such that $\ch(\exists x\,(\varphi(x)))\leq\ch(\phi)$, $\q(\exists x\,(\varphi(x)))\leq\q(\phi)$ and
$$
 \min\left\{{\sf P}\left(G(n_i,n_i^{-\alpha})\models\exists x\,(\varphi(x))\right), {\sf P}\left(G(m_i,m_i^{-\alpha})\models\neg(\exists x\,(\varphi(x)))\right)\right\}>\varepsilon/N.
$$
Set $\hat\phi_i=\exists x\,(\varphi(x,i))$. As there is only a finite number of representatives of $\cong$-equivalence classes of sentences in $\mathcal{F}$ with a fixed quantifier depth (see, e.g.,~\cite{Logic2}),
there is only a finite number of representatives of $\cong$-equivalence classes in $\{\hat\phi_i,i\in\mathbb{N}\}$ as well. Therefore, there exists a sentence $\hat\phi=\exists x\,(\varphi(x))$ and a sequence $i_j$ such that, for all $j\in\mathbb{N}$,
$$
 \min\left\{{\sf P}\left(G(n_{i_j},n_{i_j}^{-\alpha})\models\hat\phi\right), {\sf P}\left(G(m_{i_j},m_{i_j}^{-\alpha})\models\neg(\hat\phi)\right)\right\}>\varepsilon/N.
$$
So, $\alpha\in S(\hat\phi)$. $\Box$\\

Below, we state the main result of this section, which implies the following answer on Q2:

{\it the minimal number of quantifier alternations of a first order sentence with an infinite spectrum equals $3$}.

\begin{theorem}
The minimal $k$ such that there exists $\phi\in\mathcal{F}$ with infinite $S(\phi)$ and $\ch(\phi)=k$ equals $3$.
\label{minimal_changes}
\end{theorem}

{\it Proof.} By Lemma~\ref{reduce_to_one_root} and Theorem~\ref{4_or_5}, it is enough to prove that, for any $k\in\{0,1,2\}$ and any NEPNF sentence $\phi=\exists x\,(\varphi(x))\in\mathcal{F}$ with $\ch(\phi)=k$, the set $S(\phi)$ is finite. Note that $S(\phi)=S(\neg(\phi))$. Therefore, equivalently, we may prove that spectra of sentences $\forall x\,(\varphi(x))$ are finite.

Obviously, $k\in\{0,1\}$ are subcases of $k=2$. However, below we consider $k=0$, $k=1$ alone for the sake of convenience.\\

Let $\phi_H\in\mathcal{F}$ be an existence sentence which expresses the property of containing ad induced subgraph isomorphic to $H$.\\

\subsection{No alternations}

Let $\ch(\phi)=0$, where $\phi=\exists x\,(\varphi(x))\in\mathcal{F}$ is an NEPNF sentence. Obviously, there exists a finite set $\mathcal{G}$ of graphs such that $G\models\phi$ if and only if in $G$ there is an induced subgraph which is isomorphic to some $H\in\mathcal{G}$. We get
$$
 \phi\cong\bigvee_{H\in\mathcal{G}}(\phi_H).
$$
By Theorem~\ref{existence_and_extension}, either $\rho:=\min_{H\in\mathcal{G}}\{\rho(H)\}>0$ and $S(\phi)\subset\{1/\rho\}$, or $\rho=0$ and $S(\phi)=\varnothing$.

\subsection{One alternation}

Let $\ch(\phi)=1$, where
$$
\phi=\forall y_1\ldots\forall y_s\exists x_1\ldots\exists x_m\,( \varphi(y_1,\ldots,y_s,x_1,\ldots,x_m))\in\mathcal{F}
$$
has the quantifier depth $s+m$. Obviously, there exists a finite set $\mathcal{G}$ of graphs on a set of vertices $\{a_1,\ldots,a_s\}$ and, for each $A\in\mathcal{G}$, there exists a finite set $\mathcal{H}(A)$ of graphs on a set of vertices $\{a_1,\ldots,a_s,b_1,\ldots,b_m\}$ such that
\begin{itemize}
\item for any $A\in\mathcal{G}$ and $B\in\mathcal{H}(A)$, $A=B|_{\{a_1,\ldots,a_s\}}$,
\item $G\models\phi$ if and only if, for any distinct vertices $y_1,\ldots,y_s\in V(G)$, there exist distinct vertices $x_1,\ldots,x_m\in V(G)$ ($x_j\neq y_i$) and graphs $A\in\mathcal{G}$, $B\in\mathcal{H}(A)$ such that the map $f:B\rightarrow G|_{\{y_1,\ldots,y_s,x_1,\ldots,x_m\}}$, $f(a_i)=y_i$, $f(b_j)=x_j$, is an isomorphism.
\end{itemize}
Let all graphs $A_1,\ldots,A_M$ of $\mathcal{G}$ be ordered in a way such that
\begin{equation}
\rho_1:=\rho(A_1)\geq\ldots\geq\rho(A_M)=:\rho_M.
\label{graphs_densities_notation}
\end{equation}
For each $i\in\{1,\ldots,M\}$, let $\rho^i=\min\{\rho(B,A_i),B\in\mathcal{H}(A_i)\}.$

Suppose that $1/\alpha$ is not equal to any of $\rho_i,\rho^i$, $i\in\{1,\ldots,M\}$. If there is a graph on the set of vertices $\{a_1,\ldots,a_s\}$ which does not belong to $\mathcal{G}$ such that its maximal density is less than $1/\alpha$, then, by Theorem~\ref{existence_and_extension}, $G(n,p)\models\neg(\phi)$ (a.a.s.). Suppose that the above property does not hold. This implies that $\rho_M=0$. Set $\rho_0=\infty$, $1/\rho_0=0$ and $1/\rho_M=\infty$. Let $i_0\in\{0,1,\ldots,M-1\}$ be chosen in the following way: $1/\rho_{i_0}<\alpha<1/\rho_{i_0+1}$. If for some $i\in\{i_0+1,\ldots,M\}$ the inequality $\rho^i>1/\alpha$ holds, then, by Theorem~\ref{existence_and_extension}, $G(n,p)\models\neg(\phi)$ (a.a.s.). Otherwise, $G(n,p)\models\phi$ (a.a.s.). Thus, $S(\phi)\subseteq\{1/\rho_1,\ldots,1/\rho_M,1/\rho^1,\ldots,1/\rho^M\}$, and so $|S(\phi)|<\infty$.

\subsection{Two alternations}

In this case, it is not enough to define sets of graphs as above. We divide the proof into four parts. Only the first part ``Transition to sets of graphs'' is similar to the previous cases.

\subsubsection{Transition to sets of graphs}
\label{beginning}

Let a sentence
$$
\phi=\forall y_1\ldots\forall y_s\exists x_1\ldots\exists x_m\forall w_1\ldots\forall w_r\,( \varphi(y_1,\ldots,y_s,x_1,\ldots,x_m,w_1,\ldots,w_r))\in\mathcal{F}
$$
has the quantifier depth $s+m+r$.

Obviously, there exists a set of vertices $\Sigma=\Sigma_a\sqcup\Sigma_b\sqcup\Sigma_c$, where  $\Sigma_a=\{a_1,\ldots,a_s\}$, $\Sigma_b=\{b_1,\ldots,b_m\}$, $\Sigma_c=\{c_1,\ldots,c_r\}$, and
\begin{itemize}
\item[---] a finite set of graphs $\mathcal{G}$ on the set of vertices $\Sigma_a$,

\item[---] for each $A\in\mathcal{G}$, a finite set $\mathcal{H}(A)$ on the set of vertices $\Sigma_a\sqcup\Sigma_b$,

\item[---] for each $A\in\mathcal{G}$ and $B\in\mathcal{H}(A)$, a finite set of graphs $\mathcal{K}(B)$ on the set of vertices $\Sigma$,
\end{itemize}
such that the following properties hold.
\begin{itemize}
\item For any $A\in\mathcal{G}$, $B\in\mathcal{H}(A)$, $C\in\mathcal{K}(B)$, we have $A=B|_{\Sigma_a}$, $B=C|_{\Sigma_a\sqcup\Sigma_b}$.
\item $G\models\phi$ if and only if for any pairwise distinct $y_1,\ldots,y_s$ from $V(G)$ there exist pairwise distinct $x_1,\ldots,x_m$ from $V(G)\setminus\{y_1,\ldots,y_s\}$ such that for any pairwise distinct $w_1,\ldots,w_r$ from $V(G)\setminus\{y_1,\ldots,y_s,x_1,\ldots,x_m\}$ the graph $G$ has the property $P(y_1,\ldots,y_s,x_1,\ldots,x_m,w_1,\ldots,w_r)$ (which is defined below).
\end{itemize}

Let us say that {\it $G$ has the property $P(y_1,\ldots,y_s,x_1,\ldots,x_m,w_1,\ldots,w_r)$}, if
there exist graphs $A\in\mathcal{G}$, $B\in\mathcal{H}(A)$, $C\in\mathcal{K}(B)$ such that the map $f:C\rightarrow G|_{\{y_1,\ldots,y_s,x_1,\ldots,x_m,w_1,\ldots,w_r\}}$ which preserves the orders of the vertices ($f(a_i)=y_i$, $f(b_j)=x_j$, $f(c_h)=w_h$) is an isomorphism.\\

Theorem~3 from~\cite{Zhuk0} implies that $\alpha\notin S(\phi)$ for any $\alpha<\frac{1}{s+m+r-2}$. Therefore, for any positive integer $N$, the set of numbers from $S(\phi)$ with a numerator at most $N$ is finite. So, we may assume that the numerator of $\alpha$ is large enough. As in the case of one alternation, we assume that any graph on the set of vertices $\Sigma_a$ with a maximal density less than $1/\alpha$ belongs to $\mathcal{G}$.

\subsubsection{Dense neighbourhood and its structure}
\label{dense_neighborhood}

Let $\Gamma$ be an arbitrary graph on a set of vertices $V$ with the following property. There is $A\in\mathcal{G}$ and pairwise distinct vertices $y_1,\ldots,y_s\in V$ such that the map $A\rightarrow \Gamma|_{\{y_1,\ldots,y_s\}}$ which preserves the orders of the vertices is an isomorphism.

Let $Y_0=\Gamma|_{\{y_1,\ldots,y_s\}}$. For each $i\geq 0$, let us construct an induced subgraph $Y_{i+1}$ of $\Gamma$ on the union of $V(Y_i)$ with some additional vertices (for a step $\tilde i$, this process halts, set $Y=Y_{\tilde i}$). For a step $i$ the process {\it does not halt}, if there exists a subgraph $W\subset\Gamma$ such that $W\supset Y_i$, $v(W)-v(Y_i)\leq r$ and $\rho(W,Y_i)>1/\alpha$. For such a graph $W$, set $Y_{i+1}=W$.

{\it The graph $Y=Y(\Gamma;y_1,\ldots,y_s)$ is constructed.} Before proceeding with the next part of the proof, let us study a structure of $Y$ and introduce some notations for describing this structure.

\begin{itemize}
\item Let $\mathcal{U}=\mathcal{U}(A)=\{B_1,\ldots,B_{\beta}\}$ be the set of all graphs $B$ on the set of vertices $\Sigma_a\cup\Sigma_b$ such that $B|_{\Sigma_a}=A$. Obviously, $\beta=2^{C_m^2+sm}$.

\item Let $x^0_1,\ldots,x^0_m$ be arbitrary vertices which are not in $V(Y)$ (and even not necessarily in $V$).

\item Let $\ell\in\{1,\ldots,\beta\}$, $\tilde m\in\{0,\ldots,m\}$. Consider the set $\mathcal{X}_{\ell,\tilde m}$ of all collections of vertices $x_1,\ldots,x_{\tilde m}\in V(Y)\setminus\{y_1,\ldots,y_s\}$ such that there exists a graph $W$ on the set of vertices $V(Y)\cup\{x^0_{\tilde m+1},\ldots,x^0_m\}$ and an isomorphism $f:B_{\ell}\to W|_{\{y_1,\ldots,y_s,x_1,\ldots,x_{\tilde m},x^0_{\tilde m+1},\ldots,x^0_m\}}$ which preserves the orders of the vertices ($f(a_i)=y_i$, $f(b_j)\in\{x_j,x^0_j\}$).

\item For each $\ell\in\{1,\ldots,\beta\}$, $\tilde m\in\{0,\ldots,m\}$, $(x_1,\ldots,x_{\tilde m})\in\mathcal{X}_{\ell,\tilde m}$, consider the set $\mathcal{S}_{\ell}(x_1,\ldots,x_{\tilde m})$ of all graphs $W$ on the set of vertices $V(Y)\cup\{x^0_{\tilde m+1},\ldots,x^0_m\}$ such that $W|_{V(Y)}=Y$, $\rho(W,Y)<1/\alpha$ and the map $f:B_{\ell}\to W|_{\{y_1,\ldots,y_s,x_1,\ldots,x_{\tilde m},x^0_{\tilde m+1},\ldots,x^0_m\}}$ which preserves the orders of the vertices ($f(a_i)=y_i$, $f(b_j)\in\{x_j,x^0_j\}$) is an isomorphism. Moreover, for each $W\in\mathcal{S}_{\ell}(x_1,\ldots,x_{\tilde m})$ consider the set $\mathcal{N}_{\ell}(W;x_1,\ldots,x_{\tilde m})$ of all graphs $C$ on the sets of vertices $\Sigma_a\cup\Sigma_b\cup\{c_1,\ldots,c_{\tilde r}\}$, where $\tilde r\leq r$, such that there exists
    a subgraph $Z\subset W$ containing the vertices $y_1,\ldots,y_s$, $x_1,\ldots,x_{\tilde m}$, $x^0_{\tilde m+1},\ldots,x^0_m$, and the following two properties hold.  First, there exist
    vertices $w_1,\ldots,w_{\tilde r}\in V(Y)$ and an isomorphism $f:C\rightarrow Z$ which preserves the orders of the vertices ($f(a_i)=y_i$, $f(b_j)\in\{x_j,x^0_j\}$, $f(c_h)=w_h$). Second,
    $$
    \rho\left(Z,Z|_{\{y_1,\ldots,y_s,x_1,\ldots,x_{\tilde m},x^0_{\tilde m+1},\ldots,x^0_m\}}\right)>1/\alpha.
    $$

\item For each $\ell\in\{1,\ldots,\beta\}$, denote by
    $(\mathcal{N})_{\ell}[Y;y_1,\ldots,y_s]$ a maximal set of pairwise distinct sets among $\mathcal{N}_{\ell}(W;x_1,\ldots,x_{\tilde m})$, $W\in\mathcal{S}_{\ell}$.

\end{itemize}

The vector $\mathbf{N}=((\mathcal{N})_1[Y;y_1,\ldots,y_s],\ldots,(\mathcal{N})_{\beta}[Y;y_1,\ldots,y_s])$ {\it defines the structure} of $Y$.

\subsubsection{Existence of a bounded graph with the same structure}

Let $\{y_1,\ldots,y_s\}$ be an arbitrary set of vertices, and $A\in\mathcal{G}$.

Consider an arbitrary graph $\Gamma$ which contains the vertices $y_1,\ldots,y_s$ such that the map $A\rightarrow \Gamma|_{\{y_1,\ldots,y_s\}}$ (preserving the orders of the vertices) is an isomorphism. Let $\ell\in\{1,\ldots,\beta\}$ (where $\beta$ is the cardinality of $\mathcal{U}(A)=\{B_1,\ldots,B_{\beta}\}$).  Determine the vector $(\mathcal{N})_{\ell}:=(\mathcal{N})_{\ell}[Y(\Gamma;y_1,\ldots,y_s);y_1,\ldots,y_s]$.
Let $\mathbf{Y}=\mathbf{Y}(\Gamma;y_1,\ldots,y_s)$ be the set of all graphs $Y$ such that $Y|_{\{y_1,\ldots,y_s\}}=\Gamma|_{\{y_1,\ldots,y_s\}}$, and
$(\mathcal{N})_{\ell}=(\mathcal{N})_{\ell}[Y;y_1,\ldots,y_s]$ for all $\ell\in\{1,\ldots,\beta\}$. Let the graph $Y_{\min}(\mathbf{Y};y_1,\ldots,y_s)$ has a minimal number of vertices among the graphs in the set
$$
 \{Y\in\mathbf{Y}:\quad\forall\tilde Y\in\mathbf{Y}\,(\rho(\tilde Y)\geq \rho(Y))\}
$$
(and, of course, belongs to this set).

Note that the set $\mathbf{Y}(\Gamma;y_1,\ldots,y_s)$ is defined by the vector $\mathbf{N}=((\mathcal{N})_1,\ldots,(\mathcal{N})_{\beta})$ only. Therefore, for the vertices $y_1,\ldots,y_s$ there exist only finite set of pairwise distinct sets $\mathbf{Y}(\cdot;y_1,\ldots,y_s)$. So, the set of pairwise distinct graphs $Y_{\min}(\cdot;y_1,\ldots,y_s)$ is finite. Let
$$
Y_{\min}^1(y_1,\ldots,y_s),\ldots,Y_{\min}^{\theta}(y_1,\ldots,y_s)
$$
be all such graphs.

\subsubsection{Finiteness of the spectrum}

Recall that the numerator of the irreducible fraction $\alpha=\frac{R}{P}$ is large enough (see Section~\ref{beginning}). So, we assume that $R>\max\{s+m+r,N\}$, where
$$
N:=\max\left\{v(Y_{\min}^1(y_1,\ldots,y_s)),\ldots,v(Y_{\min}^{\theta}(y_1,\ldots,y_s))\right\}.
$$
Note that $N$ does not depend on a choice of $y_1,\ldots,y_s$.\\

Theorems~\ref{existence_and_extension},~\ref{double-extension} imply that a.a.s. the random graph $G(n,n^{-\alpha})$ has the following properties:

\begin{itemize}
\item[G1] for any $H$ with $\rho(H)>1/\alpha$ and $v(H)\leq s+r(sP+1)$, there is no subgraph isomorphic to $H$;
\item[G2] for any $H\subset G$ with $v(G)\leq s+m+r$ and $\rho(G,H)<1/\alpha$, there is the $(G,H)$-extension property;
\item[G3] for any $H\subset G$ with $v(G)\leq \max\{N,m+s+rsP\}$, $\rho(G,H)<1/\alpha$ and set $\mathcal{W}$ of graphs $W$ on a fixed set of vertices such that
\begin{itemize}
\item $G\subset W$, $1\leq v(W)-v(G)\leq r+m$, $\rho(W,G)>1/\alpha$,
\item $W\setminus G$ is connected,
\item there are edges between $W\setminus G$ and $G$ in $W$,
\end{itemize}
there is the $(\mathcal{W},G,H)$-double-extension property.
\end{itemize}

Let us prove that if the graphs $\Gamma,\Upsilon$ have the properties G1, G2 and G3, then either $\phi$ is true for both of them, or $\phi$ is false for both of them. This would imply that $\alpha\notin S(\phi)$.\\

Assume that $\Gamma\models\neg(\phi)$, $\Upsilon\models\phi$. By the property G1, a maximal density of any subgraph of $\Gamma$ on $s$ vertices is less than $1/\alpha$. All graphs on the set of vertices $\Sigma_a$ with such a maximal density are in $\mathcal{G}$ (see Section~\ref{beginning}). Therefore, there exist $A\in\mathcal{G}$ and pairwise distinct $y_1,\ldots,y_s\in V(\Gamma)$ such that the map $A\rightarrow \Gamma|_{\{y_1,\ldots,y_s\}}$ which preserves the orders of the vertices is an isomorphism, and $\Gamma$ with distinguished vertices $y_1,\ldots,y_s$ {\it does not have} the property (EXT), which is defined below.
\begin{center}
(EXT): {\it there exist pairwise distinct $x_1,\ldots,x_m\in V(\Gamma)\setminus\{y_1,\ldots,y_s\}$ such that for any pairwise distinct $w_1,\ldots,w_r\in V(\Gamma)\setminus\{y_1,\ldots,y_s,x_1,\ldots,x_m\}$ there exist graphs $B\in\mathcal{H}(A)$, $C\in\mathcal{K}(B)$ and an isomorphism $f:C\rightarrow \Gamma|_{\{y_1,\ldots,y_s,x_1,\ldots,x_m,w_1,\ldots,w_r\}}$ which preserves the orders of the vertices ($f(a_i)=y_i$, $f(b_j)=x_j$, $f(c_h)=w_h$).}
\end{center}

Construct the graph $Y=Y(\Gamma;y_1,\ldots,y_s)$ as it is done in Section~\ref{dense_neighborhood}. Let us prove that $v(Y)\leq s+rsP$. Assume that the opposite inequality is true. By the definition of $Y$, there is a subgraph $X\subset Y$ on at most $s+r(sP+1)$ vertices such that, for some $v_1,\ldots,v_{sP+1}\in\{1,\ldots,r\}$,
$$
 \rho(X)\geq\frac{(1/\alpha)v_1+\ldots+(1/\alpha)v_{sP+1}+\frac{sP+1}{R}}
 {s+v_1+\ldots+v_{sP+1}}=\frac{1}{\alpha}+\frac{1}{R(s+v_1+\ldots+v_{sP+1})}>\frac{1}{\alpha}.
$$
This contradicts the property G1. So, $v(Y)\leq s+rsP$, and, therefore, $\rho(Y)\leq 1/\alpha$.

Consider the graph $Y_{\min}=Y_{\min}(\mathbf{Y}(\Gamma;y_1,\ldots,y_s);y_1,\ldots,y_s)$. We have $\rho(Y_{\min})\leq\rho(Y)\leq 1/\alpha$. As $v(Y_{\min})\leq N<R$, the equality $\rho(Y_{\min})=1/\alpha$ is impossible, and so $\rho(Y_{\min})<1/\alpha$.\\

By the property G3, in $\Upsilon$ there is an induced subgraph $Y^{\Upsilon}\cong Y_{\min}$ such that in $\Upsilon$ there is no subgraph $W\supset Y^{\Upsilon}$ with $v(W)-v(Y^{\Upsilon})\leq r+m$ and $\rho(W,Y^{\Upsilon})>1/\alpha$. Let $f:Y_{\min}\rightarrow Y^{\Upsilon}$ be an isomorphism. Set $f(y_i)=y^{\Upsilon}_i$, $i\in\{1,\ldots,s\}$. As $\Upsilon\models\phi$, $\Upsilon$ with distinguished vertices $y^{\Upsilon}_1,\ldots,y^{\Upsilon}_s$ has the property (EXT). Let $x_1^{\Upsilon},\ldots,x_m^{\Upsilon}\in V(\Upsilon)\setminus\{y^{\Upsilon}_1,\ldots,y^{\Upsilon}_s\}$ and $B\in\mathcal{H}(A)$ be such that
for any pairwise distinct $w^{\Upsilon}_1,\ldots,w^{\Upsilon}_r\in V(\Upsilon)\setminus\{y^{\Upsilon}_1,\ldots,y^{\Upsilon}_s,x^{\Upsilon}_1,\ldots,x^{\Upsilon}_m\}$ there exist a graph $C\in\mathcal{K}(B)$ and an isomorphism $g:C\rightarrow \Upsilon|_{\{y^{\Upsilon}_1,\ldots,y^{\Upsilon}_s,x^{\Upsilon}_1,\ldots,x^{\Upsilon}_m,w^{\Upsilon}_1,\ldots,w^{\Upsilon}_r\}}$ which preserves the orders of the vertices ($g(a_i)=y^{\Upsilon}_i$, $g(b_j)=x^{\Upsilon}_j$, $g(c_h)=w^{\Upsilon}_h$).

From the property G1 it follows that
$$
\rho\left(\Upsilon|_{V(Y^{\Upsilon})\cup\{x_1^{\Upsilon},\ldots,x_m^{\Upsilon}\}},Y^{\Upsilon}\right)<1/\alpha
$$
if at least one of the vertices $x_1^{\Upsilon},\ldots,x_m^{\Upsilon}$ is not in $Y^{\Upsilon}$. Indeed, there is no equality, because $v\left(\Upsilon|_{V(Y^{\Upsilon})\cup\{x_1^{\Upsilon},\ldots,x_m^{\Upsilon}\}}\right)-
v\left(Y^{\Upsilon}\right)\leq m$. Let $x_1,\ldots,x_{\tilde m}\in Y^{\Upsilon}$, $x_{\tilde m+1},\ldots,x_m\in V(\Upsilon)\setminus V(Y^{\Upsilon})$, where $\tilde m\in\{0,1,\ldots,m\}$. From the property G3, the definitions of $Y$ and $Y_{\min}$ it follows that there exist vertices $x_1,\ldots,x_{\tilde m}\in V(Y)$, $x_{\tilde m+1},\ldots,x_m\in V(\Gamma)\setminus V(Y)$ such that the following properties hold.

\begin{itemize}

\item[Q1] There exists an isomorphism $f:B\rightarrow\Gamma|_{\{y_1,\ldots,y_s,x_1,\ldots,x_m\}}$ which preserves the orders of the vertices ($f(a_i)=y_i$, $f(b_j)=x_j$).

\item[Q2] There is no $W\subset\Gamma$ such that $W\supset\Gamma|_{V(Y)\cup\{x_1,\ldots,x_m\}}$,
    $$
    v(W)-v\left(\Gamma|_{V(Y)\cup\{x_1,\ldots,x_m\}}\right)\leq r\quad\text{and}\quad \rho\left(W,\Gamma|_{V(Y)\cup\{x_1,\ldots,x_m\}}\right)>1/\alpha.
    $$

\item[Q3] Let $C$ be a graph on a set of vertices $\{a_1,\ldots,a_s,b_1,\ldots,b_m,c_1,\ldots,c_{\tilde r}\}$ (where $\tilde r\leq r$). Let $Z\subseteq\Gamma$ be a graph consisting of the vertices $y_1,\ldots,y_s,x_1,\ldots,x_m$ and some vertices $w_1,\ldots,w_{\tilde r}\in V(Y)$. Moreover, let the map $f:C\rightarrow Z$ which preserves the orders of the vertices ($f(a_i)=y_i$, $f(b_j)=x_j$, $f(c_h)=w_h$) be an isomorphism, and
    $$
    \rho\left(Z,Z|_{\{y_1,\ldots,y_s,x_1,\ldots,x_m\}}\right)>1/\alpha.
    $$
    Then, in $\Upsilon$ there is a subgraph $Z^{\Upsilon}$ consisting of the vertices $y^{\Upsilon}_1,\ldots,y^{\Upsilon}_s$, $x^{\Upsilon}_1,\ldots,x^{\Upsilon}_m$ and some vertices $w^{\Upsilon}_1,\ldots,w^{\Upsilon}_{\tilde r}\in V(Y^{\Upsilon})$ such that the map $f:C\rightarrow Z^{\Upsilon}$ which preserves the orders of the vertices ($f(a_i)=y^{\Upsilon}_i$, $f(b_j)=x^{\Upsilon}_j$, $f(c_h)=w^{\Upsilon}_h$) is an isomorphism.

\end{itemize}

By our assumption, there exist $w_1,\ldots,w_r\in V(\Gamma)\setminus\{y_1,\ldots,y_s,x_1,\ldots,x_m\}$ such that for any $C\in\mathcal{K}(B)$ the map $f:C\rightarrow \Gamma|_{\{y_1,\ldots,y_s,x_1,\ldots,x_m,w_1,\ldots,w_r\}}$ which preserves the orders of the vertices ($f(a_i)=y_i$, $f(b_j)=x_j$, $f(c_h)=w_h$) is not an isomorphism. If
\begin{equation}
\rho\left(\Gamma|_{\{y_1,\ldots,y_s,x_1,\ldots,x_m,w_1,\ldots,w_r\}},    \Gamma|_{\{y_1,\ldots,y_s,x_1,\ldots,x_m\}}\right)<1/\alpha,
\label{sparse}
\end{equation}
then by the property G2 in $\Upsilon$ there are vertices $w^{\Upsilon}_1,\ldots,w^{\Upsilon}_r$ such the the map
\begin{equation}
f:\Gamma|_{\{y_1,\ldots,y_s,x_1,\ldots,x_m,w_1,\ldots,w_r\}}\to   \Upsilon|_{\{y_1^{\Upsilon},\ldots,y_s^{\Upsilon},x_1^{\Upsilon},\ldots,x_m^{\Upsilon},
w_1^{\Upsilon},\ldots,w_r^{\Upsilon}\}}
\label{map}
\end{equation}
which preserves the orders of the vertices ($f(y_i)=y_i^{\Upsilon}$, $f(x_j)=x_j^{\Upsilon}$, $f(w_h)=w_h^{\Upsilon}$) is an isomorphism --- a contradiction.\\

If $w_1,\ldots,w_r\in V(\Gamma)\setminus V(Y)$, then Inequality~(\ref{sparse}) holds (there is no equality, because $v\left(\Gamma|_{\{y_1,\ldots,y_s,x_1,\ldots,x_m,w_1,\ldots,w_r\}}\right)-    v\left(\Gamma|_{\{y_1,\ldots,y_s,x_1,\ldots,x_m\}}\right)=r<R$).

If $w_1,\ldots,w_r\in V(Y)$ and
$$
\rho\left(\Gamma|_{\{y_1,\ldots,y_s,x_1,\ldots,x_m,w_1,\ldots,w_r\}},    \Gamma|_{\{y_1,\ldots,y_s,x_1,\ldots,x_m\}}\right)>1/\alpha,
$$
then, the definition of $Y^{\Upsilon}$ implies the existence of vertices $w^{\Upsilon}_1,\ldots,w^{\Upsilon}_r$ such that the map~(\ref{map}) which preserves the orders of the vertices is an isomorphism --- a contradiction.

Finally, let some (not all) of the vertices $w_1,\ldots,w_r$ be in $V(Y)$ (say, $w_1\ldots,w_{\tilde r}\in V(Y)$, $w_{\tilde r+1},\ldots,w_r\in V(\Gamma)\setminus V(Y)$). In $Y^{\Upsilon}$ there exist vertices $w^{\Upsilon}_1,\ldots,w^{\Upsilon}_{\tilde r}$ such that the map $f:\Gamma|_{\{y_1,\ldots,y_s,x_1,\ldots,x_m,w_1,\ldots,w_{\tilde r}\}}\to   \Upsilon|_{\{y_1^{\Upsilon},\ldots,y_s^{\Upsilon},x_1^{\Upsilon},\ldots,x_m^{\Upsilon},
w_1^{\Upsilon},\ldots,w_{\tilde r}^{\Upsilon}\}}$ which preserves the orders of the vertices is an isomorphism. Moreover,
$$
\rho\left(\Gamma|_{\{y_1,\ldots,y_s,x_1,\ldots,x_m,w_1,\ldots,w_r\}},    \Gamma|_{\{y_1,\ldots,y_s,x_1,\ldots,x_m,w_1,\ldots,w_{\tilde r}\}}\right)<1/\alpha.
$$
By the property G2, in $\Upsilon$ there exist vertices $w^{\Upsilon}_{\tilde r+1},\ldots,w^{\Upsilon}_r$ such that the map~(\ref{map}) which preserves the orders of the vertices is an isomorphism --- a contradiction. $\Box$\\

\section{Spectra of formulas with small quantifier depths}
\label{depth}

Theorem~\ref{minimal_changes} answers the second question of Section~\ref{laws}. In this section we partially answer the first and the third questions.

\subsection{Counting quantifier alternations}
We do not have a complete answer on the third question. However, in our second main result, we get a new lower bound on the minimal quantifier depth of PNF sentence with an infinite spectrum.

\begin{theorem}
The minimal $q$ such that there exists a $\PNF$ sentence $\phi\in\mathcal{F}$ with infinite $S(\phi)$ and $\q(\phi)=q$ is at least 5.
\label{PNF_spectra}
\end{theorem}

The proof is based on the statement on Ehrenfeucht game which is given below. For a positive integer $k$, consider a set $\tilde S(k)$ of $\alpha>0$ such that there exist $\varepsilon>0$ and increasing sequences $n_i,m_i$ of positive integers such that, for any $i\in\mathbb{N}$,
$$
{\sf P}\left(\text{Spoiler has a winning strategy in EHR}\left(G(n_i,n_i^{-\alpha}),
G(m_i,m_i^{-\alpha}),k,k-1\right)\right)>\varepsilon^2.
$$

\begin{lemma}
The set $\tilde S(4)\cap(1/2,1)$ is finite.
\label{strategy}
\end{lemma}

{\it Proof}. {\bf Case 1}. Let $p=n^{-\alpha}$, $\alpha\in(1/2,10/19)$.\\

Let $x_1,x_2,x_3$ be vertices of an arbitrary graph $G$. For any $i,j\in\{\{0\},\mathbb{N}\}$, we say that $(x_1,x_2,x_3)$ has the type $(i,j)$, if a number of common neighbors of $x_1,x_3$ (which are not adjacent to $x_2$) is in $i$, and a number of common neighbors of $x_2,x_3$ (which are not adjacent to $x_1$) is in $j$. Introduce a linear order $\leq$ on the set $\mathcal{I}$ of all pairs of elements from $\{\{0\},\mathbb{N}\}$ in the following way: $(\{0\},\{0\})\leq(\{0\},\mathbb{N})\leq(\mathbb{N},\{0\})\leq(\mathbb{N},\mathbb{N})$.

For any vertices $x_1,x_2$, denote by $n(x_1,x_2)$ the number of all pairs of adjacent common neighbors of $x_1,x_2$. Denote the set of all common neighbors of $x_1,x_2$ by $N(x_1,x_2)$. Denote the set of all common neighbors $x_3$ of $x_1,x_2$ such that $x_1,x_2,x_3$ have no common neighbors by $U(x_1,x_2)$.\\

We say that a graph has the {\it triangle property}, if, for any $s\in\{0,1,2\}$, any vertex $x_1$, any $x,y\in\mathcal{I}$ and any $\delta\in\{\sim,\nsim\}$, there is a vertex $x_2$ in the graph such that
\begin{itemize}
\item $x_1\delta x_2$,
\item $n(x_1,x_2)\leq 1$,
\item there is no $K_4$ containing $x_1,x_2$,
\item if $n(x_1,x_2)=1$, then $|U(x_1,x_2)|=\min\{s,1\}$,
\item if $n(x_1,x_2)=0$, then $|U(x_1,x_2)|=s$,
\item for any $x_3\in U(x_1,x_2)$, $(x_1,x_2,x_3)$ has the type $x$,
\item for any $x_3\in N(x_1,x_2)\setminus U(x_1,x_2)$, $(x_1,x_2,x_3)$ has the type $y$.
\end{itemize}
By Theorem~\ref{double-extension}, a.a.s. $G(n,p)$ has the triangle property. Moreover, by Theorem~\ref{double-extension}, a.a.s. $G(n,p)$ has the {\it sparse extension property}, which is described below. For any $m\geq 1$ and any distinct vertices $v_1,\ldots,v_m$, there are vertices $z_1,z_2$ such that
\begin{itemize}
\item $z_1$ is adjacent to $v_1$ and not adjacent to any of $v_2,\ldots,v_m$, $z_2$ is not adjacent to any of $v_1,\ldots,v_m$,

\item for any $i\in\{1,\ldots,m\}$, $s\in\{1,2\}$, $z_s\neq v_i$ and $z_s$ has no common neighbors with $v_i$.
\end{itemize}
Finally, by Theorem~\ref{double-extension}, a.a.s., in $G(n,p)$, there exists a vertex $x_1$ such that
\begin{itemize}
\item there is no $K_4$ containing $x_1$,
\item for any vertex $x_2$, $n(x_1,x_2)\leq 1$
\end{itemize}
(in such a case, we say that a graph has the {\it sparse subgraph property}).\\

Let $G,H$ be graphs with the triangle property, the sparse extension property and the sparse subgraph property. Let us describe a winning strategy of Duplicator in $\EHR(G,H,4,3)$.

{\bf In the first round}, Spoiler chooses, say, an arbitrary vertex $v_1\in V(G)$. Duplicator chooses an arbitrary vertex $u_1\in V(H)$ such that there is no $K_4$ in $H$ containing $u_1$ and, for any vertex $u_2$, $n(u_1,u_2)\leq 1$. Such a vertex exists because $H$ has the sparse subgraph property.

{\bf In the second round}, Spoiler chooses a vertex $u_2\in V(H)$. If the set $N(u_1,u_2)\setminus U(u_1,u_2)$ is non-empty, then denote by $y\in\mathcal{I}$ the least element of the set of types of $(u_1,u_2,u_3)$ over all $u_3\in N(u_1,u_2)\setminus U(u_1,u_2)$. If the set $U(u_1,u_2)$ is non-empty, then denote by $x\in\mathcal{I}$ the least element of the set of types of $(u_1,u_2,u_3)$ over all $u_3\in U(u_1,u_2)$.

Consider two cases.

\begin{enumerate}

\item $u_1\sim u_2$. Duplicator chooses $v_2\in V(G)$ such that
\begin{itemize}
\item $v_1\sim v_2$,
\item there is no $K_4$ containing $v_1,v_2$ in $G$,
\item for any $s\in\{0,1\}$, $v_1,v_2$ have exactly $s$ common neighbors if and only if $u_1,u_2$ have exactly $s$ common neighbors,
\item if $u_1,u_2$ have 2 common neighbors, then $v_1,v_2$ have  exactly 2 common neighbors,
\item if $N(u_1,u_2)\neq\varnothing$, then the types of $(v_1,v_2,v_3)$ equal $x$ for all common neighbors $v_3$ of $v_1,v_2$.
\end{itemize}
\item $u_1\nsim u_2$. Duplicator chooses $v_2\in V(G)$ such that
\begin{itemize}
\item $v_1\nsim v_2$,
\item $n(v_1,v_2)=n(u_1,u_2)$,
\item if $n(u_1,u_2)=1$, then the types of $(v_1,v_2,v_3^1)$, $(v_1,v_2,v_3^2)$ equal $y$, where $v_3^1\sim v_3^2$ are common neighbors of $v_1,v_2$,
\item if $n(u_1,u_2)=1$ and $U(u_1,u_2)\neq\varnothing$, then $U(v_1,v_2)=\{v_3\}$ and the type of $(v_1,v_2,v_3)$ equals $x$,
\item if $n(u_1,u_2)=0$ and $|U(u_1,u_2)|\geq 2$, then $U(v_1,v_2)=\{v_3^1,v_3^2\}$ and the types of $(v_1,v_2,v_3^1),$ $(v_1,v_2,v_3^2)$ equal $x$,
\item if $n(u_1,u_2)=0$ and $|U(u_1,u_2)|=1$, then $U(v_1,v_2)=\{v_3\}$ and the type of $(v_1,v_2,v_3)$ equals $x$,
\item if $N(u_1,u_2)=\varnothing$, then $N(v_1,v_2)=\varnothing$.
\end{itemize}
\end{enumerate}
Such a vertex exists because 1) after the first round, there is no $K_4$ containing $u_1$ and $n(u_1,u_2)\leq 1$ for all $u_2$; 2) $G$ has the triangle property.

{\bf In the third round}, Spoiler chooses a vertex $v_3\in V(G)$. If $v_3\sim v_1,v_3\sim v_2$, then Duplicator chooses a vertex $u_3\in V(H)$ such that
\begin{itemize}
%\item $u_3\sim u_1$, $u_3\sim u_2$,
\item if $v_3\in U(v_1,v_2)$, then $u_3\in U(u_1,u_2)$ and $(u_1,u_2,u_3)$ has the type $x$,
\item if $v_3\in N(v_1,v_2)\setminus U(v_1,v_2)$, then $u_3\in N(u_1,u_2)\setminus U(u_1,u_2)$ and $(u_1,u_2,u_3)$ has the type $y$.
\end{itemize}
Otherwise, Duplicator chooses a vertex $u_3\in V(H)$ such that
\begin{itemize}
\item $v_1\sim v_3$ if and only if $u_1\sim u_3$,
\item $v_2\sim v_3$ if and only if $u_2\sim u_3$,
\item for any $j\in\{1,2\}$, the vertices $u_j,u_3$ have no common vertices.
\end{itemize}
Such a vertex exists because $H$ has the sparse extension property.

{\bf In the fourth round}, Spoiler chooses a vertex $u_4\in V(H)$.

Obviously, if $u_4$ is a common neighbor of $u_1,u_2,u_3$, then $u_1\nsim u_2$ and $u_3\in N(u_1,u_2)\setminus U(u_1,u_2)$. Therefore, $v_3\in N(v_1,v_2)\setminus U(v_1,v_2)$. So, there exists a common neighbor $v_4\in V(G)$ of $v_1,v_2,v_3$.

Assume that $u_4$ is not a common neighbor of $u_1,u_2,u_3$. If $u_4\in N(u_1,u_2)$, then $v_1,v_2$ have at least $1$ common neighbor. So, if $u_3\notin N(u_1,u_2)$, there is $v_4\in N(v_1,v_2)$ such that $v_4\neq v_3$. If $u_3\in N(u_1,u_2$, then $v_1,v_2$ have at least $2$ common neighbors, and so there is $v_4\in N(v_1,v_2)$ such that $v_4\neq v_3$ as well. If $u_4\in N(u_1,u_3)$ (or $u_4\in N(u_2,u_3)$), then $u_3\in N(u_1,u_2)$ and $(u_1,u_2,u_3)$, $(v_1,v_2,v_3)$ has the same type. Therefore, there exists a vertex $v_4$ such that $v_4\in N(v_1,v_3)$ (or $v_4\in N(v_2,v_3)$).

In all the above cases, Duplicator chooses $v_4$.

Finally, if $u_4$ is adjacent to at most one vertex of $u_1,u_2,u_3$, then Duplicator has a winning strategy because $G$ has the sparse extension property.\\

{\bf Case 2} Let $p=n^{-\alpha}$, $\alpha\in(10/19,1)$ be rational and not equal to any fraction $a/b$ with $a\leq 20$. Note that there is only a finite number of forbidden fractions $a/b$. Moreover, let $q$ be the denominator of $\alpha$.

Let $H_2\subset H_1$, $V(H_2)=\{a_1,\ldots,a_m\}$, $V(H_1)\setminus V(H_2)=\{b_1,\ldots,b_{\ell}\}$. We say that a graph $G$ has {\it a $t$-generic $(H_1,H_2)$-extension property} if for any its distinct vertices $\x=(x_1,\ldots,x_{m})$ there exist distinct vertices $\y=(y_1,\ldots,y_{\ell})$ such that
\begin{itemize}
\item $\forall i,j\in\{1,\ldots,\ell\}$ $(y_i\sim y_j)\Leftrightarrow(b_i\sim b_j)$,
\item $\forall i\in\{1,\ldots,m\},j\in\{1,\ldots,\ell\}$ $(x_i\sim y_j)\Leftrightarrow(a_i\sim b_j)$,
\item if for some $\z=(z_1,\ldots,z_{s})$ with $s \leq t$ the pair $(G_{\x\sqcup\y\sqcup\z},G_{\x\sqcup\y})$ is $\alpha$-rigid, then there are no edges between the $\y$'s and the $\z$'s.
\end{itemize}

Is the above conditions are satisfied for some $\y$, we say that the pair $(G|_{\x\sqcup\y}, G|_{\x})$ is {\it a $t$-generic $(H_1,H_2)$-extension}.

Let $\mathcal{S}$ be a set of all graphs $G$, that satisfy the following properties.

\begin{itemize}
\item[(1)] There exists a vertex $x$ in $G$ such that there are no subgraphs $W\supset X$ with $x\in V(W)$, $\rho(W,(\{x\},\varnothing))>1/\alpha$ and $v(W)\leq 21$.
\item[(2)] For any graphs $H_2\subset H_1$ with $\rho(H_1,H_2)<1/\alpha$, $v(H_1)\leq 22$, $G$ has the 1-generic $(H_1,H_2)$-extension property.
%\item[(3)] In $G$, there is no copy of a graph $A$ with $\rho(A)>1/\alpha$ and $v(A)\leq q+1$.
\end{itemize}

By Theorems~\ref{existence_and_extension} and~\ref{double-extension}, $\lim_{n\rightarrow\infty}{\sf P}(G(n,p)\in\mathcal{S})=1$, and it is sufficient to describe a Duplicator's winning strategy in $\EHR(G,H,4,3)$ for $G,H\in\mathcal{S}$.\\

{\bf In the first round}, Spoiler chooses, say, a vertex $v_1\in V(G)$. %Consider a maximal sequence $(\{v_1\},\varnothing)=G_0\subset\ldots\subset G_K$ of subgraphs in $G$ such that, for each $i\in\{1,\ldots,K\}$, $f_{\alpha}(G_i,G_{i-1})<0$, $v(G_i)-v(G_{i-1})\leq 21$. By (3), $K\leq q$ and $\rho(G_K)<1/\alpha$ (there is no equality, because). Therefore, by (1), in $H$ there is a subgraph $H_K$ such that $G_K\cong H_K$ and there is no $H_{K+1}\supset H_K$ in $H$ with $\rho(H_{K+1},H_{K})>1/\alpha$, $v(H_{K+1})-v(H_{K})\leq 21$. Let $\varphi_1:G_K\rightarrow H_K$ be an isomorphism. Duplicator chooses $u_1=\varphi_1(v_1)$.
By the property (1), there is a vertex $u_1\in V(H)$ such that in $H$ there is no subgraph $W$ with $u_1\in V(W)$, $v(W)\leq 21$, $\rho(W,(\{u_1\},\varnothing))>1/\alpha$.

%Set $X_1=G_K$, $Y_1=H_K$.

{\bf In the second round}, Spoiler chooses a vertex $u_2\in V(H)$. Consider a maximal sequence of graphs $H|_{\{u_1,u_2\}}=H_0\subset H_1\subset\ldots\subset H_L\subset H$ with each $v(H_{i})-v(H_{i-1})=1$ and $\rho(H_i,H_{i-1})>1/\alpha$. Note that $L\leq 19$. Indeed, if $L>19$, then
$$
 \rho(H_{20},(\{u_1\},\varnothing))\geq\frac{40}{21}>\frac{19}{10}>1/\alpha,
$$
that is impossible by the choice of $u_1$.

By the choice of $u_1$, we have $\rho(H_L,(\{u_1\},\varnothing))\leq 1/\alpha$. Moreover, $\alpha$ is not equal to any fraction $a/b$ with $a\leq 20$, hence, the inequality is strict: $\rho(H_{L},(\{u_1\},\varnothing))<1/\alpha$. Set $Y=H_L$. By (2), there exists a subgraph $X$ in $G$ such that $Y\cong X$, there is an isomorphism $f:Y\rightarrow X$ such that $f(u_1)=v_1$, and there is no subgraph $W\subset G$ such that $X\subset W$, $v(W)=v(X)+1$, $\rho(W,X)>1/\alpha$.

Duplicator chooses $v_2=f(u_2)$.

%Set $X_2=G_L\cup X_1$, $Y_2=H_L\cup Y_1$.

{\bf In the third round}, Spoiler chooses a vertex $v_3\in V(G)$. Consider two cases.

If $v_3\in V(X)$, then Duplicator chooses $u_3=f^{-1}(v_3)$. If after that Spoiler chooses a vertex $u_4\in V(Y)$, then Duplicator chooses $v_4=f(u_4)$, and she wins. If Spoiler chooses $u_4\notin V(Y)$, then $\rho(H|_{\{u_1,u_2,u_3,u_4\}},H|_{\{u_1,u_2,u_3\}})<1/\alpha$. So, the property (2) implies the existence of $v_4\in V(G)$ such that $v_4\sim v_i$ if and only if $u_4\sim u_i$ for all $i\in\{1,2,3\}$. Duplicator chooses such a $v_4$ and wins.

If $v_3\notin V(X)$, then $\rho(G|_{V(X)\cup\{v_3\}},X)<1/\alpha$. So, by (2) there exists a vertex $u_3$ in $H$ such that there exists an isomorphism $\tilde f:G|_{V(X)\cup\{v_3\}}\rightarrow H|_{V(Y)\cup\{u_3\}}$, $\tilde f(v_i)=u_i$ for $i\in\{1,2,3\}$. Moreover, there is no subgraph $W\subset H$ such that $H|_{V(Y)\cup\{u_3\}}\subset W$, $v(W)=v(Y)+1$, $\rho(W,Y)>1/\alpha$. Duplicator chooses that $u_3$. Denote $\tilde X=G|_{V(X)\cup\{v_3\}}$, $\tilde Y=H|_{V(Y)\cup\{u_3\}}$. If in the fourth round Spoiler chooses a vertex $u_4\in \tilde Y$, then Duplicator chooses $\tilde f^{-1}(u_4)$ and wins. If Spoiler chooses $u_4\notin \tilde Y$, then $(H|_{\{u_1,u_2,u_3,u_4\}},H|_{\{u_1,u_2,u_3\}})<1/\alpha$. In this case Duplicator's winning strategy is the same as in the first case. $\Box$\\

{\it Proof of Theorem~\ref{PNF_spectra}}. From Theorem~\ref{minimal_changes}, it follows that it is enough to prove that $|S(\phi)|<\infty$ if $\phi$ is in PNF, $\q(\phi)=4$, $\ch(\phi)=3$.

Let $\phi\in\mathcal{F}$ be a PNF sentence such that $\q(\phi)=4$, $\ch(\phi)=3$. Let $\alpha\in S(\phi)$. Obviously, there exist $\varepsilon>0$ and sequences $n_i,m_i$ such that, for any $i\in\mathbb{N}$,
$$
\min\left\{{\sf P}\left(G(n_i,n_i^{-\alpha})\models\phi\right),
{\sf P}\left(G(m_i,m_i^{-\alpha})\models\neg(\phi)\right)\right\}>\varepsilon.
$$
By Lemma~\ref{Ehren_3},
$$
{\sf P}\left(\text{Spoiler has a winning strategy in EHR}\left(G(n_i,n_i^{-\alpha}),
G(m_i,m_i^{-\alpha}),4,3\right)\right)>\varepsilon^2.
$$
Therefore, $\alpha\in\tilde S(\phi)$. By Lemma~\ref{strategy}, $|\tilde S(4)\cap(1/2,1)|<\infty$. Moreover, the random graph $G(n,n^{-\alpha})$ obeys zero-one 4-law if $\alpha<1/2$, and the set $S(\phi)\cap(1,\infty)$ is finite (see Section~\ref{laws}). Therefore, $S(\phi)=S(\phi)\cap[1/2,\infty)$ is finite. $\Box$\\

Note that as the formula~(\ref{PNF_example}) with an infinite spectrum is in PNF, Theorem~\ref{PNF_spectra} implies that a minimal quantifier depth of a PNF sentence with an infinite spectrum is in $\{5,6,7,8\}$.

Finally, it is easy to see that Lemma~\ref{strategy} and Theorem~\ref{minimal_changes} have a more general corollary which is given below.

\begin{theorem}
Let $\phi\in\mathcal{F}$, $\q(\phi)=4$. If either all paths of $F(\phi)$ starting in a root have $3$ labels alternations, or all paths of $F(\phi)$ starting in a root have at most $2$ labels alternations, then $|S(\phi)|<\infty$.
\label{4_changes}
\end{theorem}

From Theorem~\ref{4_changes}, we get that if there exists a sentence $\phi=\exists x\,\varphi(x)\in\mathcal{F}$ with $\q(\phi)=4$ and an infinite spectrum, then $F(\phi)$ has both types of paths starting in the root: with maximal number of labels alternations and with less number of labels alternations.

%So, we still do not have an answer on the first question of Section~\ref{laws}, but we know much more about sentences with $\q(\phi)=4$ and finite spectra.

\subsection{There are only two possible limit points}

We do not have an answer on the first question of Section~\ref{laws}, but we bound the set of possible limit points of $S(4)$ by two points.

\begin{theorem}
There are no limit points of $S(4)$, except, possibly $1/2$ and $3/5$.
\label{except_0.5_and_0.6}
\end{theorem}

The scheme of the proof of Theorem~\ref{except_0.5_and_0.6} is the following. First, we introduce some auxiliary constructions, that we exploit in our proof. Second, we restrict the set of possible limit points of $S(4)$ by $\{1/2,3/5,2/3,3/4\}$. Finally, we prove that all limit points of $S(4)$ are less than $2/3$.

Let $H\subset G$ be arbitrary graphs. We say that the pair $(G,H)$ is {\it $\alpha$-safe}, if $\rho(G,H)<1/\alpha$. We say that $(G,H)$ is {\it $\alpha$-rigid}, if $(e(G)-e(S))/(v(G)-v(S))>1/\alpha$ for each $H\subseteq S\subset G$.

%% LEMMA: NONSAFE EXTENSION HAS RIGID SUBEXTENSION

\begin{lemma}
\label{has_rigid}
Let $\alpha$ be not equal to any fraction $a/b$ with $a \leq v(G,H)$.
If $(G,H)$ is not $\alpha$-safe then there exists a graph $S$ such that $H \subset S \subseteq G$ and $(S,H)$ is $\alpha$-rigid.
\end{lemma}

\begin{proof}
As $(G,H)$ is not $\alpha$-safe, there exists $H \subset S \subseteq G$ with $(e(S)-e(H))/(v(S)-v(H)) \geq 1/\alpha$ (obviously, the case of equality is not possible because of the restriction on $\alpha$). Consider a minimal such $S$, and let us prove that $(S,H)$ is $\alpha$-rigid. Indeed, if there is some $H \subset S' \subset S$, such that $(e(S)-e(S'))/(v(S)-v(S')) \leq 1/\alpha$, then $(e(S')-e(H))/(v(S')-v(H)) \geq 1/\alpha$, and $S$ is not minimal.
\end{proof}

Let $\alpha>0$ be a fixed number, let $G$ be an arbitrary graph, and let $U \subset V(G)$ be an arbitrary set of its vertices. Consider a maximal sequence of induced subgraphs $G|_{U}=G_0\subset G_1\subset\ldots\subset G_L\subset G$, where each $(G_{i},G_{i-1})$ is $\alpha$-rigid, and each $v(G_{i},G_{i-1})\leq t$. Such a sequence is called an {\it $\alpha$-rigid $t$-chain}. Let us show, that the last graph $G_L$ of this sequence is the same for all possible chains. Indeed, if $G_0\subset G_1\subset \ldots \subset G_L$ and $G_0 \subset H_1 \subset \ldots \subset H_K$ are $\alpha$-rigid $t$-chains, then $G_0 \subset G_1 \subset \ldots \subset G_L \subset [G_L \cup H_1] \subset \ldots \subset [G_L \cup H_K]$ is also an $\alpha$-rigid $t$-chain. The graph $G_L$ is called the {\it $t$-closure of $U$}, and is denoted by $\cl_{t}(U)$.

%% LEMMA: FINITE CLOSURE FOR BIG NUMERATORS

\begin{lemma}
Let $a_1$, $b_1$, $a_2$, $b_2$, $t$, $C$ be fixed positive integers such that numerators of all fractions $a/b \in [a_1/b_1, a_2/b_2)$ are strictly greater than $t$. Then there exists a constant $\widetilde{C}$ such that for any $\alpha \in [a_1/b_1, a_2/b_2)$ a.a.s. $G(n,n^{-\alpha})$ has the following property. For each subset $U \subset V(G(n,n^{-\alpha}))$ with $|U|\leq C$, the inequality $v(\cl_t(U))\leq\widetilde{C}$ holds.
\label{finite_closure}
\end{lemma}

\begin{proof}
Denote by $K$ the maximum value among all fractions $a/b$ with $a\leq t$ and $a/b < a_1/b_1$. Set $\varepsilon=a_1/b_1-K$. Let us prove that $\widetilde{C}=C+tC/\varepsilon$ satisfies the statement of Lemma. Fix an arbitrary $\alpha \in [a_1/b_1, a_2/b_2]$ and consider an arbitrary $\alpha$-rigid $t$-chain $G|_U=G_0\subset G_1\subset\ldots\subset G_L$. If $L>C/\varepsilon$ then
$$
\frac{1}{\rho(G_L)}=\frac{C+\sum_{i=1}^{L}v(G_i,G_{i-1})}{\sum_{i=1}^{L}e(G_i,G_{i-1})} \leq \frac{C+K\sum_{i=1}^{L}e(G_i,G_{i-1})}{\sum_{i=1}^{L}e(G_i,G_{i-1})} \leq K+\frac{C}{L} < K+\varepsilon = \frac{a_1}{b_1} < \frac{1}{\alpha}.
$$
By Theorem~\ref{existence_and_extension}, a.a.s. there are no copies of $G_L$ in $G(n,n^{-\alpha})$. As there are finitely many possible $G_L$, a.a.s. $L \leq C/\varepsilon$ for every $U$, and hence $\widetilde{C}=C+tC/\varepsilon \geq v(\cl_t(U))$.
\end{proof}

Consider the modification
$$
\EHR[(G;x_1,\ldots,x_m),(H;y_1,\ldots,y_m),k]
$$
of the game $\EHR(G,H,k)$, in which the first $m$ moves $x_1,\ldots,x_m \in V(G)$, $y_1,\ldots,y_m \in V(H)$ are determined beforehand, i.e. for each $i\in\{1,\ldots,m\}$ in the $i$-th round the vertices $x_i$, $y_i$ are chosen ($k-m$ rounds remain). Define the relation $\sim_k$ on the set of tuples $(G;x_1,\ldots,x_m)$ (where $G$ is an arbitrary graph and $x_1$, \ldots, $x_m$ --- its vertices) in the following way: $(G;x_1,\ldots,x_m) \sim_k (H;y_1,\ldots,y_m)$ if and only if Duplicator has a winning strategy in $\mathrm{EHR}[(G;x_1,\ldots,x_m),(H;y_1,\ldots,y_m),k]$. Note that $\sim_k$ is an equivalence relation, and the number of equivalence classes of $\sim_k$ is finite (see, e.g.,~\cite{Pikhurko}). Consider a rooted tree $T=T(G;x_1,\ldots,x_m)$ whose vertices represents all possible sequences $(x_1,\ldots,x_l)$, where $k\geq l \geq m$, and $x_{l+1},\ldots,x_k$ are arbitrary vertices from $V(G)$. The root of $T$ is $(x_1,\ldots,x_m)$. For every vertex $(x_1,\ldots,x_l)$ of $T$, its children are all of the form $(x_1,\ldots,x_l,x_{l+1})$. Note that each sequence of moves in the considered Ehrenfeucht game corresponds in a natural way to a path in $T$ from the root to one of the leaves. We may assume that players moves a pebble (that is initially in the root) along the edges of $T$, instead of pebbling new vertices in $G$.

Let us make some modifications of the tree, that would not change the result of the game. Let $(x_1,\ldots,x_l)$ be an arbitrary vertex of $T$. We say that two of its descendants $(x_1,\ldots,x_l,x_{l+1})$ and $(x_1,\ldots,x_l,x_{l+1}')$ are equivalent, if
$$
(G;x_1,\ldots,x_l,x_{l+1}) \sim_k (G;x_1,\ldots,x_l,x'_{l+1}).
$$
Note that if we remove one of two equivalent vertices from the tree (together with its subtree), then the result of Ehrenfeucht game would not change. Make all that removals, and denote the modified tree by $\widetilde{T}$. As the quotient set of $\sim_k$ is finite, each vertex of $\widetilde{T}$ has a bounded number of descendants (this bound depends only on $k$ but not on $G$). Hence, $v(\widetilde{T})$ is also bounded by a constant, that does not depend on $G$. Consider the set of all vertices, that are represented in the sequences of $\widetilde{T}$, and denote the subgraph induced on this set of vertices by $K(G;x_1,\ldots,x_m;k)$. Obviously, $v(K(G;x_1,\ldots,x_m;k))$ is also bounded by some constant $C(k)$.

Now we are ready to prove the following lemma.

\begin{lemma}
For any $\varepsilon>0$, the intersection of $S(4)$ with the set
$$
A_{\varepsilon} \defeq [1/2+\varepsilon,3/5) \cup [3/5+\varepsilon,2/3) \cup [2/3+\varepsilon,3/4) \cup [3/4+\varepsilon,1)
$$
is finite.
\label{except_4_points}
\end{lemma}

\begin{proof}
Let us prove that, for all but a finite number of $\alpha \in A_{\varepsilon}$, Duplicator has an (a.a.s.) winning strategy in the game $\EHR(G,H,4)$, where $G \sim G(n,n^{-\alpha})$, $H \sim G(m,m^{-\alpha})$.

Let $a$ be an a.a.s. uniform (by all $\alpha\in A_{\varepsilon}$) upper bound for $v(\cl_3(U))+C(4)$, where $U \in V(G(n,n^{-\alpha}))$ is a set of cardinality $C(4)$ (such an upper bound exists by Lemma~\ref{finite_closure}). Let $C_1$ be an a.a.s. uniform upper bound for $v(\cl_3(\widetilde{U}))$, where $\widetilde{U}$ is a set of cardinality $a$. Let $C_2$ be an a.a.s. uniform upper bound for $v(\cl_3(\{x_1,x_2\})) + 3$, where $x_1,x_2\in V(G(n,n^{-\alpha}))$. Set $C\geq\max\{C_1,C_2\}$.

Let $\mathcal{A}_{\varepsilon}$ be the set of all $\alpha \in A_{\varepsilon}$, that are not equal to any fraction $s/t$ with $s \leq C$. Obviously, the set $A_{\varepsilon}\setminus\mathcal{A}_{\varepsilon}$ is finite. Fix an arbitrary $\alpha \in \mathcal{A}_{\varepsilon}$, and let $\mathcal{S}(\alpha)$ be the set of all graphs $G$ that satisfy the following properties.
\begin{itemize}
\item[P1] There are no subgraphs $W \subset G$ with $v(W) \leq C$ and $\rho(W)>1/\alpha$.
\item[P2] For every $\alpha$-safe pair $(W_1,W_2)$ with $v(W_1) \leq C$, $G$ has the $C$-generic $(W_1,W_2)$-extension property (the generic extension property is defined in the proof of Lemma~\ref{strategy}).
\item[P3] For every $x_1,x_2 \in V(G)$, $v(\cl_3(\{x_1,x_2\})) \leq C - 3$ .
\item[P4] For every $U\subset V(G)$ with $|U|\leq C(4)$, $v(\cl_3(U))+C(4) \leq a$, and for every $\widetilde{U}\subset V(G)$ with $|\widetilde{U}|\leq a$, $v(\cl_3(\widetilde{U})) \leq C$.
\end{itemize}

By Theorem~\ref{existence_and_extension} and Theorem~\ref{double-extension}, $\lim_{n\rightarrow\infty}\mathsf{P}(G(n,n^{-\alpha})\in\mathcal{S})=1$, and it is sufficient to describe a Duplicator's winning strategy in $\EHR(G,H,4)$ for $G,H\in\mathcal{S}(\alpha)$.

We will show that Duplicator can play in the first two rounds so that, for chosen vertices $x_1,x_2 \in V(G)$, $y_1,y_2 \in V(H)$, the graphs $\cl_3(\{x_1,x_2\})$ and $\cl_3(\{y_1,y_2\})$ are isomorphic. After that Duplicator wins due to the following lemma.

%% 3 CLOSURES ARE ISOMORPHIC => WIN

\begin{lemma}
Let $\alpha\in\mathcal{A}_{\varepsilon}$, $G,H\in\mathcal{S}(\alpha)$. Let $x_1,x_2 \in V(G)$ and $y_1,y_2 \in V(H)$ be vertices such that $\cl_3(\{x_1,x_2\})\cong\cl_3(\{y_1,y_2\})$. Then Duplicator has a winning strategy in $\EHR[(G;x_1,x_2),(H;y_1,y_2),4]$.
\label{3_closure}
\end{lemma}

\begin{proof}
Denote $G_2=\cl_3(\{x_1,x_2\})$, $H_2=\cl_3(\{y_1,y_2\})$. Without loss of generality, suppose that, in the third round, Spoiler chooses a vertex $x_3 \in V(G)$. If $x_3 \in G_2$, then Duplicator chooses the image of $x_3$ under an isomorphism $\varphi_2: G_2 \rightarrow H_2$. Further, if in the fourth round Spoiler chooses a vertex $x_4 \in V(G_2)$ (or $y_4 \in V(H_2)$), then Duplicator chooses $y_4 = \varphi(x_4)$ (or $x_4 = \varphi^{-1}(y_4)$) and she wins. If Spoiler chooses, say, a vertex $x_4 \in V(G) \setminus V(G_2)$, then, by the definition of $G_2$, the pair $(G|_{\{x_1,x_2,x_3,x_4\}},G|_{\{x_1,x_2,x_3\}})$ is $\alpha$-safe, and hence, by the property P2, there exists $y_4$ such that $G|_{\{x_1,x_2,x_3,x_4\}} \cong H|_{\{y_1,y_2,y_3,y_4\}}$. Duplicator chooses this $y_4$ and she wins. If, in the fourth round, Spoiler chooses a vertex $y_4 \in V(H) \setminus V(H_2)$, then Duplicator's winning strategy is analogous.

Let $x_3 \in V(G) \setminus V(G_2)$. Define the graph $G_3 \supset G_2$ in the following way.
\begin{itemize}
\item If the set $N(x_1,x_3) \setminus V(G_2)$ is nonempty, then put an arbitrary vertex from this set into $G_3$.
\item If the set $N(x_2,x_3) \setminus V(G_2)$ is nonempty, then put an arbitrary vertex from this set into $G_3$.
\item Put $x_3$ into $G_3$.
\end{itemize}
Obviously, $v(G_3,G_2) \leq 3$. By the definition of $G_2$ and Lemma~\ref{has_rigid}, the pair $(G_3,G_2)$ is $\alpha$-safe, moreover, by the property P3, $v(G_3) \leq C$, hence, by the property P2, there exists a subgraph $H_3 \subset H$, such that $(H_3,H_2)$ is a $1$-generic $(G_3,G_2)$-extension. Denote an isomorphism between $G_3$ and $H_3$ by $\varphi_3$. Duplicator chooses $y_3=\varphi_3(x_3)$. If, in the fourth round, Spoiler chooses a vertex that forms an $\alpha$-safe extension over the three previously chosen vertices, then he, obviously, loses. If he chooses a vertex $y_4 \in V(H)$ such that $(H|_{\{y_1,y_2,y_3,y_4\}},H|_{\{y_1,y_2,y_3\}})$ is $\alpha$-rigid, then $y_4 \in V(H_3)$, Duplicator chooses $x_4 = \varphi_3^{-1}(y_4)$ and she wins. If Spoiler chooses a vertex $x_4 \in V(G)$, such that the pair $(G|_{\{x_1,x_2,x_3,x_4\}},G|_{\{x_1,x_2,x_3\}})$ is $\alpha$-rigid, then, by the definition of $G_3$, there exists a vertex $x'_4 \in V(G_3)$ such that $G|_{\{x_1,x_2,x_3,x_4\}} \cong G|_{\{x_1,x_2,x_3,x'_4\}}$, Duplicator chooses $y_4 = \varphi(x'_4)$ and she wins.
\end{proof}

Let us define the Duplicator's strategy for the first two rounds. Without loss of generality, in the first round, Spoiler chooses a vertex $x_1 \in V(G)$. Set $G'_1 = K(G;x_1;4)$, $G_1 = \cl_3(V(G'_1))$. By the property P4, the size of $G_1$ is not greater than $a-C(4)$. By the property P1, $\rho(G_1) < 1/\alpha$ (the case of equality is not possible as $\alpha\in\mathcal{A}_{\varepsilon}$). Therefore, by the property P2, there exists a subgraph $H_1 \subset H$ such that $(H_1,(\varnothing,\varnothing))$ is a $C$-generic $(G_1,(\varnothing,\varnothing))$-extension. Duplicator chooses $y_1\in H$, that is the image of $x_1$ under an isomorphism $\varphi: G_1 \rightarrow H_1$. Denote $H'_1 = \varphi(G'_1)$.

If, in the second round, Spoiler chooses a vertex $x_2 \in V(G)$, we may assume that $x_2 \in V(G'_1)$. Indeed, by the definition of $G'_1$, for every $x_2 \in V(G)$ there exists an $x'_2 \in V(G'_1)$, such that $(G;x_1,x_2) \sim_4 (G;x_1,x'_2)$. Hence, if we prove that, for some $y_2 \in V(H)$, $(G;x_1,x'_2) \sim_4 (H;y_1,y_2)$, then, by transitivity, $(G;x_1,x_2) \sim_4 (H;y_1,y_2)$. Thus, let $x_2 \in V(G'_1)$. Duplicator chooses $y_2=\varphi(x_2)$. Obviously, $\cl_3(\{x_1,x_2\}) \subset \cl_3(V(G'_1))$, $\cl_3(\{y_1,y_2\}) \subset \cl_3(V(H'_1))$, and hence $\cl_3(\{x_1,x_2\})\cong\cl_3(\{y_1,y_2\})$. Duplicator wins due to Lemma~\ref{3_closure}. Analogously if Spoiler chooses a vertex $y_2 \in V(H_1)$, then Duplicator chooses $x_2=\varphi^{-1}(y_2)$ and she wins.

Suppose that Spoiler chooses a vertex $y_2 \in V(H) \setminus V(H_1)$. Denote $H'_2=H_1 \cup K(H;y_1,y_2;4)$, $H_2 = \cl_3(V(H'_2))$. By the property P4, $v(H_2) \leq C$. Then, by the definition of $H_1$, there are no $\alpha$-rigid pairs $(S,H_1)$ ($H_1 \subset S \subseteq H_2$), hence, by Lemma~\ref{has_rigid}, $(H_2,H_1)$ is $\alpha$-safe. By the property P2, there exists a subgraph $G_2 \subset G$, such that $(G_2,G_1)$ is a $3$-generic $(H_2,H_1)$-extension. Because of $G_1 = \cl_3(V(G_1))$, and $(G_2,G_1)$ is a 3-generic extension, $\cl_3(\{x_1,x_2\}) \subset G_2$. Thus, $\cl_3(\{x_1,x_2\}) \cong \cl_3(\{y_1,y_2\})$, and Duplicator wins due to Lemma~\ref{3_closure}.
\end{proof}

%% LEMMA: ALPHA > 2/3

\begin{lemma}
The set $S(4) \cap (2/3,1)$ is finite.
\end{lemma}

\begin{proof}
Let $\mathcal{A}$ be the set of all $\alpha\in(2/3,1)$ that are not equal to any fraction $s/t$ with $s \leq 2C(4)+3$. Obviously, $(2/3,1)\setminus\mathcal{A}$ is a finite set. We will prove that, for all $\alpha \in \mathcal{A}$, Duplicator has an (a.a.s.) winning strategy in the game $\EHR(G,H,4)$, where $G \sim G(n,n^{-\alpha})$, $H \sim G(m,m^{-\alpha})$. Fix an arbitrary $\alpha\in\mathcal{A}$, and let $\mathcal{S}(\alpha)$ be the set of all graphs $G$ that satisfy the following properties.
\begin{itemize}
\item[P1] There are no subgraphs $W \subset G$ with $v(W) \leq 2C(4)+3$ and $\rho(W)>1/\alpha$.
\item[P2] For every $\alpha$-safe pair $(W_1,W_2)$ with $v(W_1) \leq 2C(4)+3$, $G$ has the $C(4)$-generic $(W_1,W_2)$-extension property (the generic extension property is defined in the proof of Lemma~\ref{strategy}).
\end{itemize}

By Theorem~\ref{double-extension}, $\lim_{n\rightarrow\infty}\mathsf{P}(G(n,n^{-\alpha})\in\mathcal{S})=1$, and it is sufficient to describe a Duplicator's winning strategy in $\EHR(G,H,4)$ for $G,H\in\mathcal{S}(\alpha)$.

Without loss of generality, in the first round, Spoiler chooses a vertex $x_1 \in V(G)$. Denote $G_1 = K(G;x_1;4)$. Obviously, $v(G_1) \leq C(4)$. By the property P1, $\rho(G_1) < 1/\alpha$ (the case of equality is impossible as $\alpha\in\mathcal{A}$). By the property P2, there exists a subgraph $H_1 \subset H$ such that $(H_1,(\varnothing,\varnothing))$ is a $C(4)$-generic $(G_1,(\varnothing,\varnothing))$-extension. Duplicator chooses a vertex $y_1\in H$, that is the image of $x_1$ under an isomorphism $\varphi: G_1 \rightarrow H_1$.

Further, while Spoiler chooses only vertices from $G_1$ or $H_1$, Duplicator chooses their images under the isomorphism $\varphi$ or $\varphi^{-1}$. So, if Spoiler's moves are only in $G_1 \cup H_1$ then Duplicator wins. Let $i$ be the first round when Spoiler chooses a vertex not from $G_1 \cup H_1$. If it is a vertex $x_i \in V(G)$, then find a vertex $x'_i \in V(G_1)$ such that $(G;x_1,\ldots,x_i) \sim_4 (G;x_1,\ldots,x'_i)$. The replacement of $x_i$ by $x'_i$ would not change the result of the game. Thus, we may assume that the first vertex chosen outside $G_1 \cup H_1$ is a vertex from $H$. Consider several cases.

{\bf Case 1: $\mathbf{i=4}$}. Let $x_1,x_2,x_3 \in G_1$, $y_1,y_2,y_3 \in H_1$ be chosen in the first 3 rounds. In the fourth round, Spoiler chooses a vertex $y_4 \in V(H) \setminus V(H_1)$. By the definition of $H_1$, the pair $(H|_{V(H_1)\cup\{y_4\}},H_1)$ is $\alpha$-safe. By the property P2, there exists a vertex $x_4 \in V(G)$, such that $G|_{V(G_1)\cup\{x_4\}} \cong H|_{V(H_1)\cup\{y_4\}}$, Duplicator chooses this vertex and she wins.

{\bf Case 2: $\mathbf{i=3}$}. Let $x_1,x_2 \in G_1$, $y_1,y_2 \in H_1$ be chosen in the first two rounds, in the third round, Spoiler chooses a vertex $y_3 \in V(H) \setminus V(H_1)$. Denote $H_3=H_1 \cup K(H;y_1,y_2,y_3;4)$. Note that $v(H_3,H_1) \leq C(4)$ and $v(H_3) \leq 2C(4)$. By the definition of $H_1$ and Lemma~\ref{has_rigid}, the pair $(H_3, H_1)$ is $\alpha$-safe. By the property P2, there exists a subgraph $G_3 \subset G$ such that $(G_3,G_1)$ is a $1$-generic $(H_3,H_1)$-extension. Denote an isomorphism between $G_3$ and $H_3$ by $\varphi_3$. Duplicator chooses $x_3=\varphi^{-1}(y_3)$.

If in the fourth round, Spoiler chooses a vertex $y_4 \in V(H)$, we may assume that $y_4 \in V(H_3)$, because $H_3$ contains $K(H;y_1,y_2,y_3;4)$ as  a subgraph. Duplicator chooses $x_4=\varphi^{-1}(y_4)$ and she wins. If Spoiler chooses $x_4 \in G_3$, then Duplicator chooses $y_4=\varphi(x_4)$ and she wins.

Suppose that in the fourth round Spoiler chooses a vertex $x_4 \in V(G) \setminus V(G_3)$. If $(G|_{\{x_1,x_2,x_3,x_4\}},G|_{\{x_1,x_2,x_3\}})$ is $\alpha$-safe, then by the property P2 Duplicator has a winning move $y_4\in V(H)$. Let $(G|_{\{x_1,x_2,x_3,x_4\}},G|_{\{x_1,x_2,x_3\}})$ be $\alpha$-rigid. By the definition of $G_3$, $x_4$ can not be adjacent to any of $V(G_3,G_1)$, hence $x_1 \sim x_4$, $x_2 \sim x_4$, $x_3 \nsim x_4$. Find an $x'_4 \in V(G_1)$ such that $(G;x_1,x_2,x_4) \sim_4 (G;x_1,x_2,x'_4)$. Find an $x''_4\in V(G_1)$ such that $(G;x_1,x_2,x'_4,x_4) \sim_4 (G;x_1,x_2,x'_4,x''_4)$. Obviously, $x'_4,x''_4 \in N(x_1,x_2)$. Consider the vertices $y'_4=\varphi(x'_4)$, $y''_4=\varphi(x''_4)$. As $y'_4,y''_4 \in V(H_1)$, and $(H|_{V(H_1)\cup\{y_3\}},H_1)$ is $\alpha$-safe, $y_3$ is not adjacent to at least one of $y'_4$, $y''_4$. Duplicator chooses this non-adjacent to $y_3$ vertex and she wins.

{\bf Case 3: $\mathbf{i=2}$}. In the second round, Spoiler chooses a vertex $y_2 \in V(H) \setminus V(H_1)$. Define $H_2=H_1 \cup K(H;y_1,y_2;4)$. Note that $v(H_2,H_1) \leq C(4)$ and $v(H_2)\leq 2C(4)$. By the definition of $H_1$ and Lemma~\ref{has_rigid}, the pair $(H_2, H_1)$ is $\alpha$-safe. By the property P2, there exists a subgraph $G_2 \subset G$, such that $(G_2,G_1)$ is a $3$-generic $(H_2,H_1)$-extension. Denote an isomorphism between $G_2$ and $H_2$ by $\varphi_2$. Duplicator chooses $x_2=\varphi^{-1}(y_2)$.

If, in the third round, Spoiler chooses a vertex $y_3 \in V(H)$, then we may assume that $y_3 \in V(H_2)$. Duplicator chooses $x_3=\varphi_2^{-1}(y_3)$. Let $y_3 \in V(H_1)$, $x_3 \in V(G_1)$. In {\bf Case 2}, we have proved that $(G;x_1,x_3,x_2) \sim_4 (H;y_1,y_3,y_2)$. Therefore $(G;x_1,x_2,x_3) \sim_4 (H;y_1,y_2,y_3)$, and Duplicator wins. Consider the case $y_3 \in V(H_2) \setminus V(H_1)$, $x_3 \in V(G_2) \setminus V(G_1)$. If in the fourth round Spoiler chooses a vertex from $V(H) \cup V(G_2)$, then he obviously loses. If he chooses $x_4 \in V(G) \setminus V(G_2)$, then $(G|_{\{x_1,x_2,x_3,x_4\}},G|_{\{x_1,x_2,x_3\}})$ is $\alpha$-safe and he also loses.

If in the third round Spoiler chooses $x_3 \in V(G_2)$, then Duplicator chooses $y_3=\varphi_2(x_3)$. Above, we prove that if Spoiler chooses $y_3$ instead of $x_3$, then $x_3$ is a winning move for Duplicator. Therefore, in the current situation, $y_3$ is a winning move in response to Spoiler's $x_3$.

Suppose that in the third round Spoiler chooses a vertex $x_3 \in V(G) \setminus V(G_2)$. In an arbitrary graph, denote by $N(a_1,\ldots,a_m,\overline{b}_1,\ldots,\overline{b}_{\ell})$ the set of all vertices adjacent to all of $a_i$ and none of $b_j$. Define a set $W=W(x_1,x_2,x_3)$ in the following way.
\begin{itemize}
\item Put $x_1$, $x_2$, $x_3$ into $W$.
\item If $N(x_1,x_2,x_3) \neq \varnothing$, then put an arbitrary $v_{0} \in N(x_1,x_2,x_3)$ into $W$.
\item If $N(x_2,x_3,\overline{x}_1) \neq \varnothing$, then put an arbitrary $v_{1} \in N(x_2,x_3,\overline{x}_1)$ into $W$.
\item If $N(x_1,x_3,\overline{x}_2) \neq \varnothing$, then put an arbitrary $v_{2} \in N(x_1,x_3,\overline{x}_2)$ into $W$.
\item If $N(x_1,x_2,\overline{x}_3) \neq \varnothing$, then put an arbitrary $v_{3} \in N(x_1,x_2,\overline{x}_3)$ into $W$.
\end{itemize}

As $(G_2,G_1)$ is a $3$-generic extension, the vertices $v_0$ and $v_3$ (if they exist) are contained in $V(G_2)$. Hence, $|W \setminus V(G_2)| \leq 3$. Denote by $G_3$ the subgraph of $G$, whose set of vertices is $V(G_2) \cup W$, and the set of edges is
$$
E(G_2) \cup \left(E(G|_{W}) \setminus \bigcup_{0 \leq i < j \leq 3}\{v_i,v_j\}\right).
$$
Obviously, $v(G_3) \leq 2C(4)+3$. Consider the pair $(G_3,G_2)$.

If $(G_3,G_2)$ is $\alpha$-safe, then, by the property P2, there exists a subgraph $H_3 \subset H$ such that $(H_3,H_2)$ is a $1$-generic $(G_3,G_2)$-extension. Denote an isomorphism between $G_3$ and $H_3$ by $\varphi_3$. Duplicator chooses $y_3=\varphi(x_3)$. If, in the fourth round, Spoiler chooses a vertex that forms an $\alpha$-safe extension over the triple of previously chosen vertices, then Duplicator wins. If Spoiler chooses $x_4 \in V(G)$, that forms $\alpha$-rigid extension over $(x_1,x_2,x_3)$, then we may assume that $x_4 \in W$. Duplicator chooses $y_4=\varphi_3(x_4)$ and she wins. If Spoiler chooses $y_4 \in V(H_3)$, then Duplicator chooses $x_4=\varphi_3^{-1}(x_4)$ and wins again. If Spoiler chooses $y_4 \in V(H) \setminus V(H_3)$, that forms an $\alpha$-rigid extension over $(y_1,y_2,y_3)$, then $y_4 \nsim y_3$ as $(H_3,H_2)$ is a $1$-generic extension. Hence, $y_4 \in N(y_1,y_2,\overline{y}_3)$. In this case, there exists at least two vertices $y'_4,y''_4 \in N(y_1,y_2) \cap V(K(H;y_1,y_2;4))$, and at least one of them is not contained in $V(H_1)$ because $(H_2,H_1)$ is $\alpha$-safe. Consider the vertices $x'_4=\varphi^{-1}(y'_4)$, $x''_4=\varphi^{-1}(y''_4)$. At least one of them, say $x'_4$, is not adjacent to $x_3$ because $(G_2,G_1)$ is a $3$-generic extension. Duplicator chooses $x'_4$ and she wins.

Suppose that $(G_3,G_2)$ is not $\alpha$-safe. Let us prove the following two statements:
\begin{itemize}
\item[S1] There are no edges between $x_3$ and the vertices from $V(G_2) \setminus V(G_1)$.
\item[S2] If $v_1$ exists, then $v_1 \in V(G_1)$.
\end{itemize}

By Lemma~\ref{has_rigid}, there exists an $\alpha$-rigid pair $(S,G_2)$ ($G_2 \subset S \subseteq G_3$). Pick a maximal such $S$. There are no edges between $V(S,G_2)$ and $V(G_2,G_1)$, because $(G_2,G_1)$ is a 3-generic extension. Therefore $v_1 \notin V(S,G_2)$. If $x_3 \in V(S,G_2)$ then S1 is true, moreover $v_1$ can not belong to $V(G_3,G_2)$, otherwise, $(G|_{V(S)\cup\{v_1\}},G_2)$ is $\alpha$-rigid and $S$ is not maximal. Thus, $v_1 \notin V(G_3,G_2)$. Moreover, $v_1 \notin V(G_2,G_1)$ because $x_3\sim v_1$. Hence $v_1\in V(G_1)$ (if it exists) and S2 is true. Assume that $x_3 \notin V(S,G_2)$, then $V(S,G_2)=\{v_2\}$. If S1 is not true, then $(G|_{V(S)\cup\{x_3\}},G_2)$ is $\alpha$-rigid and $S$ is not maximal --- a contradiction. If S2 is not true, then $v_1$ must be in $V(G,G_2)$ (the case $v_1 \in V(G_2,G_1)$ is impossible due to S1). But then $(G|_{V(S)\cup\{x_3,v_1\}},G_2)$ is $\alpha$-rigid (it follows from the facts that $\alpha>2/3$ and that there are at least 3 edges: $\{v_2,x_3\}$, $\{x_3,v_1\}$, $\{v_1,x_2\}$, in  $E(G|_{V(S)\cup\{x_3,v_1\}},S)$), again $S$ is not maximal.

We have already mentioned that $v_0 \in V(G_2)$ (if it exists). The following statement is a corollary of S1.
\begin{itemize}
\item[S3] If $v_0$ exists, then $v_0 \in V(G_1)$.
\end{itemize}

Let us prove that only one vertex among $v_0$, $v_1$, $v_3$ exists. As $(G_2,G_1)$ is $\alpha$-safe, there is at most one edge between $x_2$ and $V(G_1)$, so (by S2 and S3) the vertices $v_0$ and $v_1$ could not both exist. By the same reason, if one of $v_0$, $v_1$ exists, then $v_3$ could not be in $V(G_1)$. Moreover, if $v_3 \in V(G_2,G_1)$, then $\rho(G_2|_{V(G_1)\cup\{x_2,v_3\}},G_1) \geq 3/2 > 1/\alpha$, and $(G_2,G_1)$ is not $\alpha$-safe --- a contradiction.

Suppose that there exists $v_0$. We have already proved that, for every $x'_3 \in V(G_2)$, there exists an $y_3 \in V(H)$ such that $(G;x_1,x_2,x'_3) \sim_4 (H;y_1,y_2,y_3)$. Hence by transitivity of $\sim_4$ it is sufficient to find such an $x'_3 \in V(G_2)$, that $(G;x_1,x_2,x_3) \sim_4 (G;x_1,x_2,x'_3)$. By the definition of $G_1$ and S1, there exists a vertex $x'_3 \in V(G_1)$, such that $(G;x_1,v_0,x_3) \sim_4 (G;x_1,v_0,x'_3)$. Let us prove, that $(G;x_1,x_2,x_3) \sim_4 (G;x_1,x_2,x'_3)$. Consider the set $W'=W(x_1,x_2,x'_3)$ defined in the same way as $W$, but with replacement of $x_3$ by $x'_3$. Denote the elements of $W'\setminus\{x_1,x_2,x'_3\}$ by $v'_0,\ldots,v'_3$. Note that (by the property P2) it is sufficient to show that, for each $i\in\{0,1,2,3\}$, $v_i$ exists if and only if $v'_i$ exists. Obviously, $v_0 = v'_0$, so they both exist. If $v_2$ exists, then, by the definition of $x'_3$, the set $N(x_1,x'_3)\setminus\{v_0\}$ is nonempty. As $(G_2,G_1)$ is $\alpha$-safe, and it is a $3$-generic extension, and $v_0 \in V(G_1)$, each of $N(x_1,x'_3)\setminus\{v_0\}$ is not adjacent to $x_2$, hence $v'_2$ exists (it could be any of $N(x_1,x'_3)\setminus\{v_0\}$). Analogously, if $v'_2$ exists, then $N(x_1,x_3)\setminus\{v_0\}$ is nonempty, and $v_2$ could be any vertex from this set, so $v_2$ exists. Note also that at most one vertex among $v'_0$ ,$v'_1$, $v'_3$ exists (the proof is the same as the above one for $v_0$, $v_1$, $v_3$). Thus, none of $v_1$, $v'_1$, $v_3$, $v'_3$ exists, and the equivalence $(G;x_1,x_2,x_3) \sim_4 (G;x_1,x_2,x'_3)$ is proved.

The case when $v_1$ exists could be considered literally in the same way with the replacement of $v_0$ by $v_1$ and vice versa.

Suppose that there exists $v_3$. Then $v_0$ and $v_1$ do not exist. As $(G_3,G_2)$ is not $\alpha$-safe, $x_3$ must be adjacent to $x_1$, and $v_2$ must exist. By the definition of $G_1$, there exist $x'_3,x''_3 \in V(G_1)$ such that $(G;x_1,x_3,v_2) \sim_4 (G;x_1,x'_3,x''_3)$. If $v_3 \in V(G_1)$, then at least one of $x'_3$, $x''_3$ (say $x'_3$) is not adjacent to $v_3$, because otherwise $\rho(G|_{\{x_1,x'_3,x''_3,x_3\}})=3/2>1/\alpha$, that is impossible by P1. Note that the sets $N(x_1,x_2,x'_3)$, $N(x_1,x_2,x_3)$, $N(x_2,x'_3,\overline{x}_1)$, $N(x_2,x_3,\overline{x}_1)$ are empty, and the sets $N(x_1,x'_3,\overline{x}_2)$, $N(x_1,x_3,\overline{x}_2)$, $N(x_1,x_2,\overline{x'}_3)$, $N(x_1,x_2,\overline{x}_3)$ are nonempty. Hence, $(G;x_1,x_2,x_3) \sim_4 (G;x_1,x_2,x'_3)$.

Let $v_3 \in V(G_2) \setminus V(G_1)$. Note that any $x'_3 \in V(G_1)$ is not adjacent to $v_3$, because otherwise the vertex $v_3$ would form an $\alpha$-rigid extension over $G_1$. So, it is sufficient to find such an $x'_3 \in V(G_1)$, that $x'_3 \sim x_1$, $N(x_1,x'_3) \neq \varnothing$ and $N(x_2,x'_3) = \varnothing$. If $v_2 \in V(G_1)$, then there exists a vertex $v \in V(G_1)$, such that $(G;x_1,v_2,x_3) \sim_4 (G;x_1,v_2,v)$, and there exists $u \in V(G_1)$, such that $(G;x_1,v_2,v,x_3) \sim_4 (G;x_1,v_2,v,u)$. All three vertices $v_2$, $v$, $u$ are adjacent to $x_1$, and all of them has at least one common neighbor with $x_1$. If at least one of them does not have common neighbors with $x_2$, then this vertex meets the requirements for $x'_3$. Suppose that there exist all three common neighbors $\widetilde{v} \in N(x_2,v)$, $\widetilde{u} \in N(x_2,u)$, $\widetilde{v}_2 \in N(x_2,v_2)$. All of these common neighbors must be in $V(G_2)$, because $(G_2,G_1)$ is a $3$-generic extension. If some of them are in $V(G_1)$, then $\rho(G|_{V(G_1)\cup\{x_2, v_3\}},G_1) \geq 3/2 > 1/\alpha$ --- a contradiction with the fact that $(G_2,G_1)$ is $\alpha$-safe. If $\widetilde{v},\widetilde{u},\widetilde{v}_2 \in V(G_2,G_1)$, then $\rho(G|_{V(G_1)\cup\{\widetilde{v},\widetilde{u},\widetilde{v}_2,x_2\}},G_1) \geq 6/4 > 1/\alpha$, and again we get a contradiction with the fact that $(G_2,G_1)$ is $\alpha$-safe.

Consider the last case $v_2 \notin V(G_1)$. Let $v \in V(G_1)$ be a vertex such that $(G;x_1,x_3) \sim_4 (G;x_1,v)$. Let $u \in V(G_1)$ be a vertex such that $(G;x_1,v,x_3) \sim_4 (G;x_1,v,u)$. Let $w \in V(G_1)$ be a vertex such that $(G;x_1,v,x_3,v_2) \sim_4 (G;x_1,v,u,w)$. Obviously, $v$, $u$ and $w$ are pairwise distinct, they all are adjacent to $x_1$ and all have common neighbors with $x_1$. Hence, analogously with the previous case, at least one of these vertices meets the requirements for $x'_3$. Thus, we have considered all the cases, and proved that Duplicator has a winning strategy.
\end{proof}

\end{document}